\documentclass[aos]{imsart}

\RequirePackage{amsthm,amsmath,amsfonts,amssymb}
\RequirePackage{stmaryrd, verbatim}
\RequirePackage[vlined,boxed,ruled]{algorithm2e}
\RequirePackage{mathrsfs}
\RequirePackage{cancel}
\RequirePackage[authoryear]{natbib}
\RequirePackage{algorithmicx}
\RequirePackage{algpseudocode}
\RequirePackage{wrapfig}
\RequirePackage{tikz}
\RequirePackage{xy}
\RequirePackage[hidelinks]{hyperref}
\RequirePackage{graphicx}

\usetikzlibrary{arrows,calc,fadings,decorations.pathreplacing}

\newcommand{\oop}{\text{op}}

\newcommand{\tr}{\text{tr}}
\newcommand{\vare}{\text{Var}}

\makeatletter
\newcommand*{\deq}{\mathrel{\rlap{%
\raisebox{0.3ex}{$\m@th\cdot$}}%
\raisebox{-0.3ex}{$\m@th\cdot$}}=}
\makeatother

\startlocaldefs
\theoremstyle{plain}

\newtheorem{assumption}{Assumption}

\newtheorem{theorem}{Theorem}[section]
\newtheorem{lemma}[theorem]{Lemma}
\newtheorem{proposition}[theorem]{Proposition}
\newtheorem{corollary}[theorem]{Corollary}
\theoremstyle{remark}
\newtheorem{definition}[theorem]{Definition}
\newtheorem*{example}{Example}

\newtheorem*{remark}{Remark}

\endlocaldefs

\begin{document}

\begin{frontmatter}
\title{%
Lag selection and estimation of stable parameters for\\ multiple autoregressive processes through\\
convex programming
}
\runtitle{estimation of stable parameters for\\ multiple autoregressive processes}

\begin{aug}

\author[]{\fnms{Somnath}~\snm{Chakraborty}\ead[label=e1]{somnath.chakraborty@ruhr-uni-bochum.de}\orcid{0000-0002-5369-3053}}, 
\author[]{\fnms{Johannes}~\snm{Lederer}\ead[label=e2]{johannes.lederer@ruhr-uni-bochum.de}\orcid{0000-0002-5369-3053}}
\and
\author[]{\fnms{Rainer}~\snm{von Sachs}\ead[label=e3]{rainer.vonsachs@uclouvain.be}}

\address[A]{Fakult\"{a}t f\"{u}r Mathematik, Ruhr-Universit\"{a}t Bochum, 44801 Bochum, Deutschland\printead[presep={,\ }]{e1,e2}}

\address[B]{Institut de Statistique, Biostatistique et Sciences Actuarielles, LIDAM,  UCLouvain, 1348 Louvain-la-Neuve, Belgium\printead[presep={,\ }]{e3}}
\end{aug}


\begin{abstract}
Motivated by a variety of applications,
high-dimensional time series have become an active topic of research.
In particular, several methods and finite-sample theories for individual stable autoregressive processes with known lag have become available very recently. We, instead, consider multiple stable autoregressive processes that share an unknown lag. We use information across the different processes to simultaneously select the lag and estimate the parameters. We prove that the estimated process is stable, and we establish rates for the forecasting error that can outmatch the known rate in our setting.
Our insights on the lag selection and the stability are also of interest for the case of individual autoregressive processes.
\end{abstract}

\begin{keyword}[class=MSC]
\kwd[Primary ]{62M10}
\kwd{62L99}
\kwd[; secondary ]{49M29}
\end{keyword}

\begin{keyword}
\kwd{regularised least square}
\kwd{LASSO}
\kwd{autoregressive process}
\kwd{hierarchical-group norm}
\kwd{dual norm}
\kwd{stability}
\kwd{sample complexity}
\end{keyword}

\end{frontmatter}

\section{Introduction}\label{Intro}

Today's world of acquisition of complex data in areas such diverse as macroeconomics and finance, everyday weather predictions, brain imaging, and many more, has called for intelligent model approaches that avoid needing to use potentially a (too) high number of model parameters per available sample size. Moreover, often these data are of high dimensionality - as they arise together in a panel or in the form of a multivariate vector. These stylized facts render the purpose of predicting the evolution of these data into the (near) future really challenging. To face this challenge choosing a data generating model that assumes some {\em common} underlying structure relating the different components of the observed multivariate data set will not only turn out to be advantageous but reflects the observation that the different series do not behave independently from each other - they might actually be driven by latent (i.e. unobservable) mechanism (such as a leading economic indicator, or a global climate trend, etc, often modelled by a latent factor model). Moreover,  we almost always observe {\em serial correlation} between present and past observations, which traditionally has been modelled by assuming some sort of weak dependence over time (translating into dynamic latent factor models, e.g., \cite{Forni}.

In this context, as factor modelling does not necessarily allow for component-wise prediction, the approach of (parametric) vector autoregression (VAR) has already for a long time become an overly prominent tool for modeling such multivariate time series - with in particular the idea that the common serial dependence is limited by the existence of a common maximal lag-order for all components. However, as the number of component series is increased, VAR models have the known tendency to become overparametrized. In the virtue of having to do with a high-dimensional parameter estimation problem, more recent possibilities to address this issue are {\em regularized} approaches, such as the LASSO for estimating the parameters of these models (essentially by some kind of regularised least-squares approach, see, for example, \cite{MR2755014}). This is in contrast to more traditional approaches (based mostly on information criteria for lag-order selection such as AIC, BIC, etc.) which address overparametrization by selecting a low lag order, based on the assumption of short range dependence, assuming that a universal lag order applies to all components. For a good forecast performance in a  high-dimensional context, these approaches turned out to fall behind the LASSO - which, until recently, did however not incorporate the notion of lag order selection. It has been only the recent work by \cite{MR4209452} that proposed a class of hierarchical lag structures that embed the notion of lag selection into a convex regularizer. The key modeling tool has been a group LASSO with nested groups which guarantees that the sparsity pattern of lag coefficients honors the VAR’s ordered structure. For more details on the literature on dimension reduction methods which address the VAR’s overparametrization problem we refer to Section 2 of the mentioned work by \cite{MR4209452}. A clear shortcoming, however, of this approach is the necessity to model all components of the observed multivariate time series to be of the same data length, a constraint in classical VAR-modelling that cannot be circumvented.

 Motivated by the approach of \cite{MR4209452}, in this paper, we propose a method to analyse multiple stable autoregressive processes of (potentially) {\em different lengths} in the framework of regularized LASSO, where the regularization is achieved via an overlapping group-norm that induces sparsity at the group level. Moreover, we show that, even in absence of any information on the maximum lag of the processes, the proposed framework estimates the true lag and the coefficients of the AR model. Finally, we show that the model fitted with the AR coefficients returned by this proposed method is stable.  As our results on statistical guarantees are essentially of non-asymptotic nature - interesting even in the context of observing a single time series -  we first review the (sparse) literature on those non-asymptotic results in a time series context, before we turn in more detail to the similarities and differences between our and the approach of \cite{MR4209452}.

Most of the research in time series analysis --- until recently --- focused on deriving asymptotic behaviour of the predictors. This severely restricts applicability of these results, especially in the reign of low sample-to-predictable ratio. Popular approaches to overcome this nuisance of dependency have been using the assumption of stability (leading to stationarity) of the data-generating process. For example, both \cite{MR2816348} and \cite{MR3015038} used stability in deriving the guarantees in small sample regimes; however, these works established these results under the condition that the coefficient matrices are severely norm-bounded (namely, the sum of the operator-norms of the coefficient matrices is smaller than 1), which is much stronger than stability of the process determined by those coefficients. Recently, \cite{MR3357870} made a big stride towards understanding the effect of temporal and cross-sectional dependency in the small sample regime. The underlying hypothesis in that work was that data be amenable to modelling via stable vector autoregressive process; they tracked the restrictions on the spectral domain --- as enforced by stability of the process --- and derived non-asymptotic prediction guarantees for high dimensional vector autoregressive process with Gaussian white noise innovation.
Several follow-up works (see \textit{e.g.} \cite{Wong_2020}, \cite{masini}, and the references there-in) then extended their results to the case where the innovations are heavy-tailed, and moreover, they derived guarantees assuming only the stationarity and finiteness of second moment of the  underlying process --- conditions weaker than stability. 

 In the remainder of this Introduction we go now into more details about the relation of our approach to the one of \cite{MR4209452}. Essentially, the latter contributed by finding out that algorithmically the method by \cite{MR2845676} (and its computationally faster amendment by \cite{tseng}) can be used in such a setting to address, in the presence of prior information on an upper bound $L$ of the true unknown lag order, estimation of  
 order $L$ vector-autoregressive models of dimension $M$, abbreviated  ${\mbox{VAR}}_M(L)$ in the sequel. Our work can now be seen as fitting a common (i.e. “diagonal vector”) autoregressive model to a multivariate time series of dimension $M$, i.e. a panel of $M$ observed (univariate) 
 time series of {\em in general not equal} lengths $n_m, 1 \leq m \leq M$.  We address the challenging question on how to choose, solely from the information available in the model for the observed data, a common appropriate lag order $L$ that allows us to phrase and solve our problem via a penalised LASSO approach.  We derive non-asymptotic bounds on the multivariate one-step ahead prediction error and estimate the collection of the autoregressive coefficients $\mathbf{\beta} \deq (\beta^{m}_{1}, \ldots, \beta^{m}_{L})_{1 \leq m \leq M}$ under the paradigm of sparseness. Assuming that there is a common true unknown lag order $L_0$ that generated our $M$ time series, our algorithm, akin \cite{MR4209452} (and \cite{MR2845676})   for fitting a common model is based on a modification of a hierarchical group LASSO approach:  We first determine an appropriate (minimal) upper bound $L\geq L_0$ depending essentially on the sample size $n_{\min}$ of the shortest observed time series component (and on $M$, of course) from a thorough analysis of the theoretical complexity of our group-LASSO based approach. With this appropriate $L$, necessary to embed our autoregression problem into the framework of high-dimensional multivariate regression, we transfer existing technology on LASSO estimation (with overlapping hierarchically constructed groups) to our problem. We derive statistical guarantees for the estimators  $\mathbf{\hat\beta}$,  solution of our aforementioned learning algorithm, and for the estimator $\hat{L}_0$ (essentially taken from the support of $\mathbf{\hat\beta}$).
 More specifically we deliver non-asymptotic bounds on the multivariate one-step ahead prediction error, on the estimation error of $\mathbf{\beta} $, on the false discoveries for the support of  $\mathbf{\beta}$, and quite innovatively on the stability of the fitted model. For the latter, we show that the fitted autoregressive model of order  $\hat{L}_0$  with estimated coefficients $\mathbf{\hat\beta}$ fulfils the conditions of the true model for stability (via a more explicit concept of $\varepsilon$-stability that we introduce to asses the difficulty of the statistical estimation problem). 
 
In the following paragraph we are even more explicit about the exact nature of our contributions motivated from the existing limitations of the current approaches we found in the literature.

\paragraph*{Current limitations and our contributions}
Some of the questions that are not sufficiently addressed in recent existing work on non-asymptotic time series analysis 
are as follows. 
\begin{itemize}
\item[L1]
\cite{MR4209452} illustrate that an approach based on LASSO with overlapping groups (such as in \cite{MR2845676}) can determine the component-wise lag orders $L_{m}$ of stable, high-dimensional vector autoregressive processes (including the "cross-over" lag order $L_{ij}$ of the dependence of the $i-$th component on the $j-$th component of a VAR model).
Their formulation based on (dual) convex programming is computationally attractive and intriguing more generally,
but it requires a uniform upper-bound on the $L_{i,j}$ as a parameter both in their theoretical bounds and in practice.
The current literature either ignores this issue altogether or sets those bounds based on model selection such as AIC or BIC or via Bayesian shrinkage,
whose suitability is unclear here.

\item[C1] 
We mimic \cite{MR4209452}'s methodological approach,
but we establish a suitable upper bound for the lag orders. 
Our choice on this minimal upper bound guarantees the following:
\begin{enumerate}
\item the smallness of the empirical prediction risk (see Theorem~\ref{lem:main1} and the discussion following the theorem),
    \item a resulting tuning parameter that is not too large (see Corollary~\ref{rates}), and 
    \item the restricted eigenvalue property of the data matrix to hold (Theorem~\ref{sc-07-11-1} and the Corollary~\ref{res} immediately after that).
\end{enumerate}
In this sense, our upper bound on the lag order is optimal for our setup. 


\item[L2] Real world applications often involve multiple univariate, decoupled time series.
This would translate to VAR modeling with a diagonal coefficient matrix. However, purely transferring existing results from a ${\mbox{VAR}}_M(L)$ modeling approach can be cumbersome in practice for
situations in which the number of available samples for each individual component time series is not the same:  obviously needing to chop off the samples in order to work with the minimal individual sample size could result in possibly weaker theoretical guarantees and practical performances (see below). 
\item[C2] 
Applying a hierarchical group norm enables us to derive prediction guarantees that depend on the {\em average} number of samples 
per component (instead of the minimum number of samples per model, as would be the 
case had we translated naively to a $M$-dimensional VAR model). 
More specifically, availability 
of perfect information on the true lag $L_0$ and $n_1 = L_0+T_1,\cdots,n_m = L_0+T_M$ 
(respective) number of samples for the $M$ individual components yields the 
following: the ${\mbox{VAR}}_M(L)$ translation would result in the following provable 
error (see \cite{MR4209452}), stating that with high probability 
\begin{equation}\label{intro1}
|\!|\hat{\boldsymbol{\beta}}-\boldsymbol{\beta}|\!|_2=O \left(\sqrt{\frac{\log(M^2L_0)}{n_{\min}M}}\right)\,,\end{equation} 
whereas the 
error from the algorithm we describe is of the order (see equation \eqref{vvv} below)
\begin{equation}\label{intro2}
|\!|\hat{\boldsymbol{\beta}}-\boldsymbol{\beta}|\!|_2= O\left(\sqrt{\frac{\log(ML)}
D}\right)\,.
\end{equation} Here $D=T_1+\cdots+T_M$ is the total number of "postsamples" ($T_m \deq n_m - L_0$). 
Furthermore, our strategy results in weaker dependency of 
the error on the behaviour of the reverse characteristic 
polynomial on the unit disk, unlike in the relevant ${\mbox{VAR}}_M(L)$ model translation (compare with \cite{MR3357870}).

\item[L3] Recall that a (univariate) autoregressive process $X_t=a_1X_{t-1}+\cdots+a_LX_{t-L}+U_t$ is stable if the ``reverse characteristics polynomial'' $1-a_1z-\cdots-a_Lz^L$ has no complex roots on the closed unit disk;
it is known that stability implies stationarity (see Section~\ref{prelims} below). But what about stability for {\em fitted} autoregressive models? While this question has an affirmative answer in the special case of Yule-Walker estimation of the coefficients of a univariate autoregressive process (known however to be less efficient), the question of stability or stationarity of the process reconstructed from the  parameters estimated by LASSO-based approaches does not seem to have been addressed in the literature. However, starting with a stable process to have generated the input observations of these devised algorithms, it is reasonable to expect that the reconstructed process be stable (and thus, multi-step predictions be reliable as well). 


\item[C3] We show, in Theorem~\ref{sc1-07-11:2}, that the process reconstructed from the parameters returned by our algorithm is stable when the samples available as input are generated by stable processes. As \cite{MR3357870} showed, a measure of stability for an autoregressive process is, equivalently, a boundedness criteria on the spectral density, and the boundedness in the Fourier domain translates in a sense to `smoothness' of the process in the temporal domain. Thus, as, intuitively, the stable autoregressive processes form a `smooth' subclass, it is desirable that any algorithm for learning the parameters of processes from this smooth subclass should return estimators lying in this subclass; in this paper, this is ensured by Theorem~\ref{sc1-07-11:2}. 

It is important to mention here that the results in this paper demonstrate that our overlapping group-lasso approach yields a stable process when the underlying process is stable as well. The aim of the paper is not to choose the most optimal tuning parameters or some absolute constants, but rather to show existence of these proposing reasonable candidates for such parameters/constants.


\end{itemize}

\subsection*{Organization of the paper}
The paper is organized as follows. In Section \ref{prelims}, we recall relevant definitions from existing literature, and we set notations. In Section \ref{sec4}, (1) we specify the model and formulate the learning problem as a regularized group-LASSO with overlapping group norm, where the data matrix is a block-diagonal matrix --- each block of which consists of the data matrix that treats a least-squares problem corresponding to the associated component time series; (2) we present the learning algorithm based on the group-LASSO problem. Section \ref{sec:7} contains the bulk of the technical contents, in particular, the proof of the statements of the main results, already presented at the end of Section \ref{sec4}. This Section \ref{sec:7} is divided into four subsections: Subsection \ref{pred-error} presents an oracle inequality bounding the one-step-ahead prediction error (Theorem \ref{lem:2-9-sc4}), as well as a high-probability bound on the effective noise of the model (Theorem \ref{lem:main1}). Subsection \ref{rip} starts with restricted eigenvalue bounds (Proposition \ref{sc-07-11-1}) for the blocks of the data matrix, and goes on to integrate the blockwise results to finally arrive at an estimate (Theorem \ref{prop:21-10-sc1}) of the error in estimating the AR-coefficients. Finally, combining the results from these subsections, stability of the estimated AR model (Theorem \ref{sc1-07-11:2}) has been established in Subsection \ref{stasta}. All proofs of auxiliary results are deferred to a series of Appendices.




\section{Preliminaries}\label{prelims}


\subsection{General notations}\label{section:notation1}
\begin{itemize}
\item $M_d(\mathbb F)$ denotes the ring of $d\times d$ matrices with entries in the field $\mathbb F\in \{\mathbb R,\mathbb C\}$, and for $M\in M_d(\mathbb F)$, we write $M^\top$ for the transpose; if $\mathbb F=\mathbb C$, then $M^\star$ denotes the conjugate transpose. 
\item $\mathbb D$ is the complex closed unit disk $\mathbb D\deq\{z\in\mathbb C:|z|\leq 1\}$, and its boundary is $\partial\mathbb D\deq\{z\in\mathbb C:|z|=1\}$.
\item For any integer $d>0$, if $x\in \mathbb R^d$, then $|\!|x|\!|_2=\sqrt{x^\top x}$, and $\mathbb S^{d-1}\deq\{x\in\mathbb R^d:|\!|x|\!|_2=1\}$; in particular, under standard identification $\mathbb C=\mathbb R^2$, we have $\mathbb S^1=\partial \mathbb D$. 
\item For integer $n>0$, we will denote the set $\{1,2,\cdots,n\}$ by $[n]$.
\item For a vector $\hat{\boldsymbol{\beta}}_m=(\hat{\boldsymbol{\beta}}_{m,1},\cdots,\hat{\boldsymbol{\beta}}_{m,L})$ and an integer $0<L_0\leq L$, we write \begin{equation}\label{cut-out}
\hat{\boldsymbol{\beta}}_m(L_0)\deq (\hat{\boldsymbol{\beta}}_{m,1},\cdots,\hat{\boldsymbol{\beta}}_{m,L_0})\,.
\end{equation}
\end{itemize}

\subsection{Notations for autoregressive process}\label{section:notation2}
\begin{enumerate}
\item[{\bf{Conventions:}}] We use $X,Y,Z,\dots$ to denote random variables, and $\boldsymbol{X},\boldsymbol{Y}, \boldsymbol{Z},\dots$ to denote random vectors. On the other hand, we use $a,b,c,\dots$ to denote real (or complex) constants, and $\boldsymbol{a},\boldsymbol{b}, \boldsymbol{c},\dots$ for vector-valued constants. 
\item[\bf{Notations:}]
\begin{itemize}
\item For $d$-dimensional autoregressive process \begin{align}X_t\deq A_1X_{t-1}+\cdots+A_LX_{t-L}+U_t\,,\label{eq:defn-1}\end{align} and $z\in\mathbb C$, we write \begin{align}\label{eq:poly}{\bf \mathcal A}_z\deq I-A_1z-\cdots-A_Lz^L\,.\end{align}
\item $L$ will denote an initially determined ``ad-hoc" upper-bound on the true lag of the process in (\ref{eq:defn-1}), and $L_0$ the true lag. 
\end{itemize}
\end{enumerate}

\begin{definition}[Weak stationarity]
A $d$-dimensional time series $\{X_t\}_{t\in\mathbb Z}$ is said to be {\em weakly stationary} if the following holds: $a)~\mathbb E[|\!|X_t|\!|^2_2]<\infty$ for all $t\in\mathbb Z$, $b)~\mathbb E[X_t]=\mu$ for all $t\in\mathbb Z$, and $c)$ $\mathbb E[X_tX_{t-h}^\top]=\Gamma(h)$ for all $t,h\in\mathbb Z$.
\end{definition}

\begin{definition}[Strong stationarity]
A $d$-dimensional time series $\{X_t\}_{t\in\mathbb Z}$ is said to be {\em strongly stationary} if for each integer $n>0$, and all integers $t_1,\cdots,t_n,h$, the distributions of the vectors $(X_{t_1},\cdots,X_{t_n})$ and $(X_{t_1+h},\cdots,X_{t_n+h})$ are identical.
\end{definition}

\begin{definition}[Autoregressive time series]
A $d$-dimensional time series $\{X_t\}_{t\in\mathbb Z}$ is autoregressive of lag at most $L>0$ if there are $d\times d$ matrices $A_1,\cdots,A_L$ such that \begin{align}\label{eq:defn1}X_t=A_1X_{t-1}+\cdots+A_LX_{t-L}+U_t\,.\end{align} holds for all $t\in\mathbb Z$, for some random white noise process $\{U_t\}_{t\in\mathbb Z}$.
\end{definition}

Associated to each $d$-dimensional lag-$L$ autoregressive process $\{X_t\}_{t\in\mathbb Z}$ --- as in equation (\ref{eq:defn1}) --- is the associated order-1 process ${\bf X}_t={\bf A}{\bf X}_{t-1}+{\bf U}_t$, where \begin{align}\label{eq:var}{\bf A}\deq\begin{pmatrix}A_{1\rightarrow L},A_L\\{\bf I}_{dL-d}\,,{\bf 0}\end{pmatrix}\,,\hspace{1cm}{\bf U}_t\deq\begin{pmatrix}U_t\\ {\bf 0}\end{pmatrix}\,,\end{align} and $A_{1\rightarrow L}$ is the block matrix $(A_1 \cdots A_{L-1})$.

\begin{definition}[Stability]
A $d$-dimensional lag-$L$ autoregressive process $\{X_t\}_{t\in\mathbb Z}$ --- as in equation \ref{eq:defn1} --- is said to be stable if $\det({\bf I}-{\bf A}z)\neq 0$ for $|z|\leq 1$. Equivalently, the process is stable if $\det({\mathcal A}_z)\neq 0$ 
for $|z|\leq 1$.
\end{definition}


\begin{definition}[Reverse characteristic polynomial]
The polynomial $\det({\bf \mathcal A}_z)$ is called the {\em reverse characteristic polynomial} of the process in equation \ref{eq:defn1}.
\end{definition}

We note the equality $\det({\bf I}-{\bf A}z)=\det({\bf \mathcal A}_z)$. In particular, the process in equation (\ref{eq:defn1}) is stable if and only if every eigenvalue of ${\bf A}$ is inside the open unit disk.\\

\begin{definition}[$\epsilon$-stability]\label{defn:new1}
A stable autoregressive process, as in equation \eqref{eq:defn1}, is said to be $\epsilon$-stable for an $\epsilon\in (0,1)$ if the following holds: \begin{align}\label{eq:quasi}\epsilon\leq \min_{|z|=1}|\det({\bf \mathcal{A}_z})|\leq \max_{|z|=1}|\det({\bf \mathcal{A}_z})|\leq \epsilon^{-1}\,.\end{align}
\end{definition}

\begin{remark}
By maximum modulus principle, this is equivalent to saying that $$\epsilon\leq \min_{|z|\leq 1}|\det({\bf \mathcal{A}_z})|\leq \max_{|z|\leq 1}|\det({\bf \mathcal{A}_z})|\leq \epsilon^{-1}\,.$$ The lower-bound here is a convenient quantification of the notion of stability, which demands that $\min_{|z|\leq 1}|\det({\bf \mathcal{A}_z})|>0$. Note that we also require the upper bound to derive our statistical guarantees.
\end{remark}

A well-known fact about autoregressive processes is the following: see \citet[proposition~2.1]{MR2172368} for details.

\begin{lemma}[{{Stability implies weak stationarity}}]\label{fact:stab}
A stable autoregressive process is weakly stationary.
\end{lemma}







\section{Statistical Model and Estimator}\label{sec4}
This section introduces our statistical model and estimator, and presents an algorithm to learn the parameters of the model from observed samples.

\subsection{Statistical Model}\label{stat-model}
We start with the model.
Suppose that we observe time-samples generated by $M$ univariate 
autoregressive process, for which we know a uniform upper-bound $L$ of the true lag-order. Then, we can aggregate these $M$ univariate lag at most $L$ autoregressive processes
\begin{align}\label{eq:0}
    X^1_t~&=~\beta_1^1 X_{t-1}^1+\cdots+\beta_{L}^1 X_{t-L}^1 +U^1_t\,;\nonumber\\
    &~\,\vdots\\
    X^M_t~&=~\beta_1^M X_{t-1}^M+\cdots+\beta_{L}^M X_{t-L}^M +U^M_t\nonumber\,.
\end{align}
In this paper, we work under the simplified assumption that the true lag of all the $M$ component processes is identical, namely, $L_0$, and that $L\geq L_0$ is generic; neither $L_0$ nor $L$ is known \emph{a priori}. Additionally, we assume mean-zero, Gaussian white-noise innovations; that is, for each $m\in [M]$ 
the set $\{U^m_t\}_{t\in\mathbb Z}$ consists of independent 
mean-zero, univariate Gaussians with coordinate-wise standard deviation $\sigma_m\in(0,\infty)$. Additionally, we assume that for each $t\in\mathbb Z$, the noise variables 
$U^1_t,\dots,U^M_t$ are independent.
We summarize the parameters of the model in a matrix $\Theta\in\mathbb{R}^{M\times L}$ 
via $\Theta_{ml}\,\deq\beta_l^m$ to refer to groups of parameters more easily later.
Our goal is 1.~to estimate the parameters 
of \textit{all} models simultaneously; and 2.~to assess the lag $L_0$, which is 
assumed to be the same over all $M$~processes.\\

We make the following assumption on the absolute value of the smallest $\beta$-coefficient,
which is widely known as \emph{$\beta$-min assumption} in the LASSO literature; see, for example, \cite{MR2461898}. We note that this assumption will only be needed to achieve the bound in Theorem \ref{algorithm-thm} after $\lambda$-thresholding; 
in particular, when no thresholding is employed, the analysis in this paper does not require the assumption.

\begin{assumption}[{\emph{$\beta$-min assumption}}]\label{beta-min}
There is an absolute constant $c_{\beta}>0$ such that the true autoregressive coefficient vector $\boldsymbol{\beta}$ satisfies \begin{equation}\label{beta-min1}\boldsymbol{\beta}^m_j\neq 0~\Rightarrow~\boldsymbol{\beta}^m_j\geq c_\beta\,.\end{equation}
\end{assumption}
\noindent
Broadly speaking,
this assumption ensures that the non-zero coefficients can be detected in the first place.
We now set out to define a regularizer. Let $n_j$ denote the total number of samples from $$X_t^j=\beta^j_1X^j_{t-1}+\cdots+\beta^j_LX^j_{t-L}+U^j_t\,,$$ and $T_j\deq n_j-L$. The main idea is as follows.
Suppose that $L\geq L_0$ is some integer, and for each $t\in\{- L+1,\dots,1,\dots,T_m\}$ and $m\in\{1,\dots,M\}$, we have an observation $x_t^m$ of $X_t^m$. Denote $\boldsymbol{\beta}_m\deq(\beta^m_1, \dots,\beta^m_L)^\top$ for each $m\in [M]$, and let $\boldsymbol{\beta}\deq(\boldsymbol{\beta}_1^\top,
\dots, \boldsymbol{\beta}_m^\top)^\top$.
We define $\mathcal{G}_1,\dots, \mathcal{G}_L\subset S_{ML}\deq\{1,2,\dots, M\}\times \{1,2,\dots,L\}$ by 
\begin{equation}\label{groupsdef}
    \mathcal{G}_l~\deq\{1,2,\dots,M\}\times \bigl\{l,l+1\dots,L\bigr\}
\end{equation}
for all $l\in\{1,\dots,L\}$. The groups are nested: $\mathcal{G}_1\supset \dots\supset\mathcal{G}_L$. Let $\boldsymbol{\beta}_{{\mathcal G}_l}\in \mathbb R^{M\times (L-l+1)}$ be the submatrix of $\boldsymbol{\Theta}$, consisting of columns having index larger or equal to~$l$.
We set the group norm to be 
\begin{align}\label{Neq}
|\!|\boldsymbol{\beta}|\!|_{\mathcal G}~&\deq \sum_{l=1}^L\sqrt{M(L-l+1)}~|\!|\boldsymbol{\beta}_{{\mathcal G}_l}|\!|_{\mathbb F}\,,\\ \mbox{where}\hspace{1cm}
    |\!|\boldsymbol{\beta}_{\mathcal{G}_l}|\!|_{\mathbb F}~&\deq\sqrt{
    \sum_{m=1}^{M}\sum_{j=l}^{L}|\beta_j^m|^2}\nonumber\,.
\end{align} is the Frobenius norm of $\boldsymbol{\beta}_{\mathcal{G}_l}$. We will alternatively write ${\mathcal N}(\boldsymbol{\beta})$ for the group norm $|\!|\boldsymbol{\beta}|\!|_{\mathcal G}$, for the sake of notational ease.

The overall post-sample size is denoted 
\[D\deq T_1+\cdots+T_M\,.\]

In order to estimate the coefficient vector 
$\boldsymbol{\beta}$, we propose 
solving the following constrained convex 
program

\begin{align}\label{eq:p1}
\mbox{minimize}&&\frac 1D\sum_{m=1}^M\sum_{t=1}^{T_m}
\bigl(x_t^m-\beta^m_1 x_{t-1}^m-\cdots-\beta^m_Lx_{t-L}^m\bigr)^2+\lambda
|\!|\boldsymbol{\beta}|\!|_{\mathcal G}
\,,\end{align}
with an appropriate tuning parameter $\lambda>0$.


\noindent 
The objective function can be put in a concise form.
For this, we define the vector $\boldsymbol{y}\in\mathbb{R}^D$, the matrix $X\in\mathbb{R}^{D\times(ML)}$, and the parameter $\boldsymbol{\beta}\in\mathbb{R}^{ML}$, as follows:
\begin{align}\label{data}
    \nonumber\boldsymbol{y}&\deq (x_1^1,\dots,x_{T_1}^1,x_1^2,\dots,x_{T_2}^2,\cdots,x_1^M,\dots,x_{T_M}^M)^\top\,;\nonumber \\
X&\deq
\begin{pmatrix}
x_{1-1}^1,\dots,x_{1-L}^1\\
\vdots\\
x_{T_1-1}^1,\dots,x_{T_1-L}^1\\
&x_{1-1}^2,\dots,x_{1-L}^2\\
&\vdots\\
&x_{T_2-1}^2,\dots,x_{T_2-L}^2\\
&&\ddots\\
&&&x_{1-1}^M,\dots,x_{1-L}^M\\
&&&\vdots\\
&&&x_{T_M-1}^M,\dots,x_{T_M-L}^M\\
\end{pmatrix}\,; \\
    \boldsymbol{\beta}&\deq(\beta_1^1,\dots,\beta_L^1,\beta^2_1,\dots,\beta^2_L,\cdots,\cdots,\beta_1^M,\dots,\beta_L^M)^\top\,.\nonumber
\end{align}
\noindent 
It is immediate that 
\begin{equation*}
    \sum_{m=1}^M\sum_{t=1}^{T_m}\bigl(x_t^m-\beta^m_1 x_{t-1}^m-\cdots-\beta^m_Lx_{t-L}^m\bigr)^2~=~|\!|\boldsymbol{y}-X\boldsymbol{\beta}|\!|_2^2\,.
\end{equation*}
In conclusion, the above estimation program is equivalent to
\begin{align}\label{eq:1}
    \widehat{\boldsymbol{\beta}}~\in~\operatornamewithlimits{argmin}_{\boldsymbol{\beta}\in\mathbb{R}^{ML}
    }\biggl\{\frac 1D|\!|\boldsymbol{y}-X\boldsymbol{\beta}|\!|_2^2+\lambda
    |\!|\boldsymbol{\beta}|\!|_{\mathcal{G}}
    \biggr\}\,.
\end{align}
Hence, the estimator can be cast as a modified group-lasso estimator, which means that we can use established group-lasso algorithms that allow for overlapping groups \citep{MR2845676}. 
In essence, the above estimator generalizes the \emph{elementwise} estimator $\mbox{HLag}^{\text{E}}$ in \cite{MR4209452} to multiple time series. Note that, the ordinary LASSO estimator---as well as any group-LASSO estimators with non-overlapping groups---enforces sparsity by setting coefficients to zero without paying heed to the fact that when only an upper-bound to the true-lag $L_0$ is an input to the regression---\emph{all} coefficients indexed between $L_0+1$ and $L$ are supposed to be zero before any coefficient with index smaller or equal to~$L_0$; however, the penalty obtained via the chained groups $\mathcal{G}_1\supseteq\dots\supseteq \mathcal{G}_L$ precisely achieves this feat.


\begin{example}[Regularizer]
We consider the case of two univariate lag (at most) three autoregressive processes; that is, $M=2$ and $L=3$. 
The corresponding groups are the following:
\begin{align*}
    {\mathcal G}_1&=\{(1,1), (1,2), (1,3), (2,1), (2,2), (2,3)\}\;;\\
    {\mathcal G}_2&=\{(1,2), (1,3), (2,2), (2,3)\}\;;\\
    {\mathcal G}_3&=\{(1,3), (2,3)\}\,.
\end{align*}
Thus, the group norm of $\boldsymbol{\beta}\in \mathbb R^{2\times 3}$ is
\begin{align*}|\!|\boldsymbol{\beta} |\!|_{\mathcal G}&=\sqrt6\sqrt{(\beta^1_1)^2+(\beta^2_1)^2+(\beta^1_2)^2+(\beta^2_2)^2+(\beta^1_3)^2+(\beta^2_3)^2}\\ &\hspace{2.5cm}+~\sqrt4\sqrt{(\beta^1_2)^2+(\beta^2_2)^2+(\beta^1_3)^2+(\beta^2_3)^2}\\ &\hspace{5cm}+~\sqrt2\sqrt{(\beta^1_3)^2+(\beta^2_3)^2}\,.\end{align*} 
Notice that, when the (regularized) LASSO sets a certain group (say the second group above) to~$\boldsymbol{0}$, it automatically sets all the following groups to $\boldsymbol{0}$ as well. More specifically, when the (regularized) LASSO sets $\boldsymbol{\beta}_{{\mathcal G}_l}\boldsymbol{0}$, then the hierarchical structure ${\mathcal G}_l\supseteq {\mathcal G}_{l+1}\supseteq \cdots$ means that for all $r>0$, each coordinate of $\boldsymbol{\beta}_{{\mathcal G}_{l+r}}$ comes as a coordinate of $\boldsymbol{\beta}_{{\mathcal G}_{l}}$, thus ensuring $\boldsymbol{\beta}_{{\mathcal G}_{l+r}}=\boldsymbol{0}$ for each $r\geq 0$.
\end{example}

In what follows, we use the following notations. For 
$m\in [M]$ and $l\in [L]$, and $t\leq T_m$, we 
write 
\begin{equation}
\label{notn11}
\begin{aligned}
\boldsymbol{U}^{(m)}&\deq(U^m_1,\cdots,
U^m_{T_m})^\top\,;\\  X^{(m,l)}&\deq(X^m_{1-l},\dots,
X^m_{T_m-l})^\top\,;\\  X_t^{(m)}&\deq(X^m_{t-1},\dots,
X^m_{t-L})\,;\\  X^{(m)}&\deq(X^{(m,1)},\dots,
X^{(m,L)})\,.
\end{aligned}
\end{equation}
Moreover, we will write $X^{(m)}_{< j}$ to denote 
any of the variables $X^m_{j'}$ for $j'<j$.

\subsection{Estimation Pipeline}
Learning autoregressive coefficients of multiple time series of potentially different lengths and identical true lag is more complex than just the usual group-LASSO problem, where the groups form a partition of the index set. 
From a methodological perspective, some immediate technical challenges are 1. deciding on what $L$ should be taken in the formulation of the convex problem \eqref{eq:p1}, 2. how to disentangle the dual of the group norm (in order to apply H\"older's inequality to derive oracle prediction guarantees as in subsection \ref{pred-error} below); from a practical perspective, the challenge lies in incorporating the varying number of samples into the convex problem.

We now give a high-level overview of our estimation pipeline (described below) that takes as input the multiple time series, and forms the appropriate convex problem in the form of \eqref{eq:1}, and solves this convex problem via stochastic proximal gradient method as in \cite{MR4209452}. In essence, the idea is to start looking into the data set to first find the samples corresponding to the component which has the minimum number of samples (breaking ties arbitrarily). We then use this component to find the initial input lag, $L$, as described in \eqref{L-bound}. In the next step, we solve the regularized least-squares problem in \eqref{eq:p1} --- where the penalty function is the group norm $\mathcal N$, discussed further below (see \eqref{Neq}). The rate of convergence of the procedure is quadratic in number of computational steps as discussed in \cite{MR4209452}.

\begin{algorithm} 
   \label{alg:learnar}

  \DontPrintSemicolon
  \SetKwInOut{Input}{Input}\SetKwInOut{Output}{Output}
  \Input{samples from the component stable processes, confidence parameters $A\geq 1$ and $\delta>0$, stability parameter $\epsilon\in (0,1)$}
  \Output{estimated lag $\hat{L}_0$, and estimated autoregression coefficient vector $\tilde{\boldsymbol{\beta}}$}
\begin{algorithmic}
\item[1.] Let $n_{\min}$ be the minimum number of samples from the components. Solve for $L$ (see equation \eqref{L-bound}): $$
    n_{\min}=L+84Ae\zeta^{-2}L\log\left(\frac{ML}\delta\right)\,.$$\;
    \item[2.] Use the learning algorithm in \cite{MR4209452} as subroutine to solve the convex problem in \eqref{eq:1} --- with $L$ as above; call the output 
    $\hat{\boldsymbol{\beta}}$.
    \item[3.] If ${\boldsymbol{\beta}}_m={\boldsymbol{\beta}}_{m'}$ for all $m,m'\in [M]$ (equivalently, the component time series are from identical AR process), return $\hat{L}_0\deq\max\{j:|(\hat{\boldsymbol{\beta}}_1)_j|>\lambda\}$ and $\tilde{\boldsymbol{\beta}}\deq(\hat{\boldsymbol{\beta}}'_0,\cdots,\hat{\boldsymbol{\beta}}'_0)$, where (refer Equation \eqref{cut-out} for notation) $$\hat{\boldsymbol{\beta}}_0'\deq\frac 1M\sum_{m=1}^M\hat{\boldsymbol{\beta}}_m(\hat{L}_0)\,;$$ 
          else, return \begin{align*}\hat{L}_0&\deq\max\{j:|\hat{\boldsymbol{\beta}}_j^m|>\lambda~\mbox{for some}~m\in [M]\}\,,\\ \text{and}\hspace{1cm} \tilde{\boldsymbol{\beta}}&\deq (\hat{\boldsymbol{\beta}}_1(\hat{L}_0),\cdots,\hat{\boldsymbol{\beta}}_m(\hat{L}_0))\,.\end{align*}

\end{algorithmic}
\caption{AR Coefficient Estimation Pipeline}
\end{algorithm}

\newpage
The following is the main theorem on the theoretical properties of the output of the estimation pipeline. This theorem is essentially a summary of the results contained in Section \ref{sec:7}, and will be discussed in all detail in subsections \ref{rip1} and \ref{stasta}, as indicated below.

\begin{theorem}[Main Theorem]\label{algorithm-thm}
    Let $\hat{\boldsymbol{\beta}}$ be the output of the group-LASSO in \eqref{eq:1} above, with \begin{equation*}\lambda=
{24(84Ae)^{\frac 12}\zeta^{-1}\sigma_{\max}^2}
{C_\sharp}^{\frac 32}(1+\epsilon^{-2}+\epsilon^{-4})
\sqrt{\frac{L\log\left(\frac{ML}\delta\right)}{DM}}\,,\end{equation*} where $\zeta=6^{-3}\epsilon^4$, and $L\geq L_0$ satisfying \begin{equation*}
    n_{\min}=L+84Ae\zeta^{-2}L\log\left(\frac{ML}\delta\right)\,.
\end{equation*}
Suppose that the $\beta$-min condition \eqref{beta-min1} holds with $c_\beta=\lambda$. If the total number $D$ of post-samples satisfies \begin{equation*}D\geq 
3^9\cdot (84Ae)C_\sharp^5(1+\epsilon^{-2}+\epsilon^{-4})^2\left(\frac{\sigma_{\max}}{\sigma_{\min}}\right)^4(\epsilon^3\zeta)^{-2}M^aL_0^2L^3
\log\left(\frac{ML}\delta\right)\log (2L)\,,\end{equation*} and if $T_{\min}\geq 84eA\zeta^{-2}L_0\log L$, then the following holds with high probability: 
\begin{enumerate}
\item the estimation error is given by  \begin{equation*}|\!|\hat{\boldsymbol{\beta}}-\boldsymbol{\beta}|\!|_2\leq 
\frac{81(84Ae)^{\frac 12}LL_0\sigma_{\max}^2C_{\sharp}^{\frac 32}(1+\epsilon^{-2}+\epsilon^{-4})}{\zeta\alpha\epsilon^2}\sqrt{\frac{\log\left(\frac{ML}{\delta}\right)}D}\,;\end{equation*} 

\item if $S_\lambda\deq\{j:|\hat{\boldsymbol{\beta}}_j|>\lambda\}$, then the false discovery is bounded by the following inequality: \begin{equation*}|S_\lambda\setminus \operatorname{supp}({\boldsymbol{\beta}})|~\leq \frac{243(84Ae)^{\frac 12}LL_0^{\frac 32}\sigma_{\max}^2C_{\sharp}^{\frac 32}(1+\epsilon^{-2}+\epsilon^{-4})}{\zeta\alpha\epsilon^2\lambda}\sqrt{\frac{\log\left(\frac{ML}{\delta}\right)}D}\,.\end{equation*}

\item the AR-models --- fitted with coefficients $\hat{\boldsymbol{\beta}}_0$ returned by the Algorithm 
\ref{alg:learnar} ("AR Coefficient Estimation Pipeline").  --- are stable, with high probability.
\end{enumerate}
\end{theorem}

\begin{proof}[Proof of Theorem \ref{algorithm-thm}]
    Subject to the stated $\beta$-min condition \eqref{beta-min1}, this follows immediately from Theorem \eqref{prop:21-10-sc1} in combination with Theorem \eqref{sc1-07-11:2}.
\end{proof}
        In the pipeline above, it is necessary to consider two distinct stability parameters $A\geq 1$ and $\delta$, as it is not possible to integrate them into a single parameter --- due mainly to the different number of samples from the component processes in our set-up.
    Also, we separately mention the two cases (of $\boldsymbol{\beta}_m$'s being identical (or not) for all $m$) in order to specifically emphasize that in the first case, our algorithm requires a smaller number of samples  than in the later case (which allows savings of a factor of $M$ in the sample complexity).
    

Decent algorithms like the proximal gradient method used in the subroutine above might produce small non-zero valued parameters as numerical artifacts.
One could consider other types of algorithms instead,
but the required number of computational steps could be much higher.
Moreover, those artifacts, and statistical false positives more generally, can also be controlled by standard $\lambda$-thresholding as mentioned in the algorithm.

\section{Statistical Guarantees}\label{sec:7}
This section contains the main theoretical results of this paper. 
We begin with deriving bounds for the one-step ahead prediction error, first formulated by an oracle inequality (see Theorem~\ref{lem:2-9-sc4}), which depends on the tuning parameter $\lambda$ of our least-squares penalisation approach. Then we control this tuning parameter by controlling the effective noise of our Lasso-optimisation problem. Both things together will finally yield more explicit rates
of our one-step ahead prediction error.
The second part of this section treats control of the estimated autoregressive coefficients, including control of false discovery for their support (Theorem~\ref{prop:21-10-sc1}). In the end we present our result on stability of the estimated AR-model (Theorem~\ref{sc1-07-11:2}).

To start with, we briefly recall the notion of the dual norm of our group-LASSO norm $\mathcal N$ defined above.


Our underlying space is $\mathbb R^{ML}$. Observe 
that $$\boldsymbol{\beta}\mapsto {\mathcal N}
(\boldsymbol{\beta})\deq \sum\sqrt{|{\mathcal G_l}|}
\cdot|\!|\boldsymbol{\beta}_{{\mathcal G}_l}|\!|_2$$ is a norm 
for any ${\mathcal G}=\{{\mathcal G}_1,\cdots,{\mathcal G}_l\}$ that covers $[ML]$. 
However, this norm is singular at any point 
where all the coordinates in a group vanish. Moreover, all the coordinates in all of the smaller-sized groups will vanish, which is important 
in our analysis, since we will basically care about sparse solutions. Thus, we can 
not appeal to differential techniques to get bounds on the norm, but rather need to appeal to the dual norm approach.

\begin{definition}[Dual norm]\label{dual_def}
For a norm ${\mathcal N}$ on $\mathbb R^{ML}$, the dual norm ${\mathcal N}_\star(\boldsymbol{\alpha})$ of 
$\boldsymbol{\alpha}\in\mathbb R^{ML}$ is the 
optimum solution of the following convex program: \begin{eqnarray*}\mbox{maximize}&&
\langle \boldsymbol{\alpha},\boldsymbol{\beta}\rangle\\
\mbox{subject to}&& \mathcal N(\boldsymbol{\beta})\leq 1\,.\end{eqnarray*} 
\end{definition}

The dual norm is used to encapsulate the effective noise of LASSO. More explicitly, the dual norm shows up in the form ${\mathcal N}_\star(X^\top\boldsymbol{U})$, called the \emph{effective noise} of the LASSO in equation \eqref{eq:1}. Recall that the effective noise vector $X^\top\boldsymbol{U}$ can be thought of as the ``projection'' of the noise vector $\boldsymbol{U}$ on the column space of $X$, and thus, ${\mathcal N}_\star(X^\top\boldsymbol{U})$ is a measure of the ``true'' noise present in the data.

In Appendix \ref{Dua_norm}, we obtain a generic bound on the dual norm, which will be used to obtain the statistical guarantees of this paper.




\subsection{\bf{Prediction Error}}\label{pred-error}

Here we now give the announced theoretical results on non-asymptotic bounds for the one-step ahead prediction error, first formulated by an oracle inequality (see Theorem~\ref{lem:2-9-sc4}), then in the following subsection including control of the  tuning parameter $\lambda$ by control of the effective noise of our Lasso-optimisation problem. This enables us to formulate concrete rates of the prediction error. Note that this approach delivers an explicit way of how to select $L$ (via equation \eqref{tmin}, and subsequently \eqref{L-bound}) the input parameter for Step 2 of Algorithm 
\ref{alg:learnar} ("AR Coefficient Estimation Pipeline"). 

\subsubsection{Oracle Prediction Error
}
The problem~(\ref{eq:1}) has a solution because, given any specific 
realization of the time series (thus, effectively, fixing $\boldsymbol{y}$ 
and $X$) and any $\lambda>0$, the convex function
\begin{eqnarray}
\label{eq:2}f_\lambda\left(\boldsymbol{\beta}\right) =\frac 1D |\!|\boldsymbol{y}-
X\boldsymbol{\beta}|\!|_2^2+\lambda
|\!|\boldsymbol{\beta}|\!|_{\mathcal{G}}\end{eqnarray} is continuous, 
and $$\lim_{|\!|\boldsymbol{\beta}|\!|_2\rightarrow\infty}f_\lambda
(\boldsymbol{\beta})=\infty\ .$$ 
This also shows that we can consider this as a convex program on a compact domain, because 
we should (at least in theory) be able to restrict the domain of minimization to be an 
$\ell^2$-ball of suitable radius, say $R_{r,X,\boldsymbol{y}}\in 
(0,\infty)$. Let $\hat{\boldsymbol{\beta}}$ be a solution of this convex program.

Now write the time series observations in the matrix form as 
$\mathcal X$. 
We are interested in an estimate of the 'risk', equivalently, the in-sample, 
one-step-ahead mean squared forecast error $\mathbb E[|\!|
\boldsymbol{y}-X\hat{\boldsymbol{\beta}}|\!|_2^2/D\mid \mathcal X]$. In the derivations below, we follow a well-known approach for its control, 
as appeared (for example) in \citep[Chapter~6]{MR4385517}. 
We defer the proof to Appendix \ref{lem:2-9-sc4:proof}.
\begin{theorem}[Prediction Guarantee]\label{lem:2-9-sc4}
Suppose that $\lambda\geq  \frac 2D\mathcal N_\star(X^\top\boldsymbol{U})$, where $\mathcal N(\boldsymbol{\alpha})=
\sum_{l=1}^L\sqrt{|\mathcal G_l|}\cdot
|\!|\boldsymbol{\alpha}_{(\geq l)}|\!|_2$ as in Proposition~\ref{lem:2-9-sc2}. 
Write  $\overline{\sigma}^2=D^{-1}(T_1\sigma_1^2+
\cdots+T_M\sigma_M^2)$; then \begin{eqnarray*}\frac 1D\mathbb E\left[|\!|\boldsymbol{y}-X\hat{\boldsymbol{\beta}}|\!|_2^2
\mid \mathcal X\right]
&\leq& \overline{\sigma}^2
+\min_{\alpha\in \mathbb R^{ML}}\left(\frac 1D|\!|X(\boldsymbol{\beta}-
\boldsymbol{\alpha})|\!|_2^2+2\lambda\mathcal N(\boldsymbol{\alpha})\right),
\end{eqnarray*} In particular, the following inequality holds: 
\begin{eqnarray*}\mathbb E\left[|\!|\boldsymbol{y}-X\hat{\boldsymbol{\beta}}|\!|_2^2
\mid \mathcal X\right]-\mathbb E\left[|\!|\boldsymbol{y}-X{\boldsymbol{\beta}}|\!|_2^2
\mid \mathcal X\right]
&\leq& \min_{\alpha\in \mathbb R^{ML}}\left(|\!|X(\boldsymbol{\beta}-
\boldsymbol{\alpha})|\!|_2^2+2D\lambda\mathcal N(\boldsymbol{\alpha})\right)
\,.\end{eqnarray*} If, moreover, $\lambda\geq \frac 4D\mathcal N_\star(X^\top\boldsymbol{U})$, then 
\begin{align}\label{eq:16-08-sc2}\frac 1D
\mathbb E\left[|\!|\boldsymbol{y}-X\hat{\boldsymbol{\beta}}|\!|_2^2
\mid \mathcal X\right]&\leq \overline{\sigma}^2
+\frac{\lambda}2\min\left\{3\mathcal N(\boldsymbol{\beta}
-\hat{\boldsymbol{\beta}}), 3\mathcal N(\boldsymbol{\beta})
-\mathcal N(\hat{\boldsymbol{\beta}})\right\}.\end{align} 
Consequently, 
\begin{align}\label{lem:18-10-sc5}
\frac 1D\mathbb E\left[|\!|\boldsymbol{y}-X\hat{\boldsymbol{\beta}}|\!|_2^2
\mid \mathcal X\right]\leq \overline{\sigma}^2+2\lambda\sum_{\ell=1}^{L_0}\sqrt{|{\mathcal G}_\ell|}\cdot
|\!|\boldsymbol{\beta}-\boldsymbol{\hat\beta}|\!|_{{\mathcal G}_\ell}\ .
\end{align}
\end{theorem}
\noindent
This result is in the form of standard oracle inequalities in high-dimensional statistics \citep[Chapter~6]{MR4385517}.
It shows that the estimator minimizes the one-step-ahead mean-squared forecast risk up to a complexity term that is linear in the tuning parameter and the model complexity.
Hence, the above bound yields an upper bound for the rate convergence once we can control the tuning parameter $\lambda$ via an upper bound on the effective noise.

Thanks to Theorem~\ref{lem:2-9-sc4} and Proposition~\ref{lem:2-9-sc2}, in order to find the smallest tuning parameter $\lambda$ fulfilling $\lambda\geq 4\mathcal N_\star(X^\top\boldsymbol{U})/D$
it suffices to derive a high-probability 
upper-bound on $D^{-1}L^{-\frac 12}|\!|X^\top\boldsymbol{U}|\!|_\infty$ (which is precisely the bound on the dual norm of $X^\top U$).



\subsubsection{Control of the Effective Noise for bounding the tuning parameter 
}
We can prove the following high-probability 
bound on the tuning parameter (equivalently, on the effective noise). 
Again, its proof appears in 
Appendix \ref{lem:main1:proof}. Note that this is a major 
inequality, and while the arguments are well-known, we have applied 
those arguments to the case where the sparsity enforcing 
regularizer is induced by the overlapping group norm.  



\begin{theorem}[Bound on the Effective Noise]\label{lem:main1}
Let $C_\sharp\deq T_{\max}/T_{\min}$. 
For any $\eta>0$ satisfying \begin{equation}\label{eta1}\eta \geq \frac{8C_\sharp \sigma^2_{\max}(1+\epsilon^{-2}+\epsilon^{-4})
}{M}
\,,\end{equation} 
and any $\delta>0$, if $D=T_1+\cdots+T_m$ satisfies 
\begin{equation}\label{eq:D1}D\geq \frac{8
\sigma_{\max}^2(1+\epsilon^{-2}+\epsilon^{-4})}{c_0\eta}\log\left(\frac{ML}{\delta}\right)\,,\end{equation}
then the following inequality holds: \begin{equation}\label{thm:blah}\mathbb P
\left[\frac 2D{\mathcal N}_\star(X
^\top\boldsymbol{U})
\geq \frac{3\eta}{2\sqrt L}\right]~\leq~ \delta\,.\end{equation} Here $c_0>0$ is the 
absolute constant from the Gaussian concentration inequality in proposition \ref{cor:25-10-sc20}.



\end{theorem}
\noindent
The interpretation of the above result is that for the LASSO oracle inequality (Theorem~\ref{lem:2-9-sc4}) to hold with high probability, it is sufficient to choose $\lambda$ to be just as large as $(3\eta)/(2\sqrt L)$, but it is not necessary to take it larger. Indeed, equation \eqref{thm:blah} shows that the probability that $2/D\mathcal N_\star(X^TU)$ is larger than $3\eta/2\sqrt L$ is small; thus, we may assume $2/D\mathcal N_\star(X^TU)<3\eta/2\sqrt L$, to hold with probability $1-\delta$.


The above result bounds the tails of the effective noise.
For any 
$A\geq 1$, we will now set  \begin{equation*}
\eta=~C_0\sqrt{A}C_\epsilon C_\sharp^{\frac 32}\sigma_{\max}^2\sqrt{\frac{L\log\left(\frac{ML}{\delta}\right)}{DM}}\,,
\end{equation*} for some absolute constant $C_0>0$ and parameter $C_\epsilon>0$ that depends only on the stability parameter $\epsilon$. 
Henceforth, we take \begin{equation}\label{etavalue}
\eta\deq 8(84Ae)^{\frac 12}\zeta^{-1}(1+\epsilon^{-2}+\epsilon^{-4})C_\sharp^{\frac 32}
\sigma_{\max}^2
\sqrt{\frac{L\log\left(\frac
{ML}{\delta}
\right)}
{DM
}}\,,\end{equation} which is obtained by plugging-in values of $C_0$ and $C_\epsilon$ --- as in the proof of Lemma \ref{lemma4} below --- into Equation \eqref{etavalue}; we do not attempt to optimize these constants.
\\ 


\begin{lemma}[Data-dependent selection of $L$
]\label{lemma4}
There is an absolute constant $C\in(0,\infty)$ and a parameter $C_\epsilon>0$ that depends only on the stability parameter $\epsilon>0$ such that if $\eta$ is as in \eqref{etavalue} 
and \begin{equation}\label{tmin}T_{\min}\leq 84Ae\zeta^{-2}L\log\left(\frac{ML}\delta\right)\,,\end{equation} then the inequality \eqref{eta1} holds. 
\end{lemma}
\noindent
This results shows that there is a suitable $\eta$ for our theories to hold.
The question of finding such an $\eta$ in practice will need to be discussed in more applied future work.


\begin{proof} 
We set $C_0\deq(12)^3\sqrt{84e}$, and $$C_\epsilon\deq\epsilon^{-4}(1+\epsilon^{-2}+\epsilon^{-4})\,.$$ Note that \begin{equation}\label{Csharp}C_{\sharp}T_{\min}= T_{\max}\geq~\frac DM\,.\end{equation} Now, with $\eta$ as in \eqref{etavalue}, the inequality \eqref{eta1} is ensured 
by the following sequence of inequalities: 

\begin{eqnarray*}
&(84Ae)^{\frac 12}\zeta^{-1}\sqrt{{L\log\left(\frac{ML}{\delta}\right)}}&\geq \sqrt{T_{\min}}\\ \Rightarrow\hspace{0.5cm} &(84Ae)^{\frac 12}\zeta^{-1}C_\sharp^{\frac 12}  \sqrt{{L\log\left(\frac{ML}{\delta}\right)}}&\geq \sqrt{\frac DM}\\ \Rightarrow\hspace{0.5cm} &8(84Ae)^{\frac 12}\zeta^{-1}(1+\epsilon^{-2}+\epsilon^{-4})C_\sharp^{\frac 32}\sigma_{\max}^2\sqrt{\frac{L\log\left(\frac{ML}{\delta}\right)}{DM}}&\geq \frac{8C_\sharp \sigma^2_{\max}(1+\epsilon^{-2}+\epsilon^{-4})}{M}
\,.\end{eqnarray*}
\end{proof}
The inequality \eqref{tmin} provides for the theoretical support of our choice of $L$ in formulating the penalized least squares program in \eqref{eq:1}. 

The explicit nature of the parameters 
in the proof above is crucial here, in order to satisfy both \eqref{eta1} and \eqref{eq:D1} above, as well as remaining compatible with the requirements involving the restricted eigenvalue property (as in Corollary \ref{res}) below. Note that the current LASSO-based literature often ignores combining the requirements coming from standard oracle inequality type result (Theorem \ref{lem:2-9-sc4}) and the restricted eigenvalue type results.

\begin{remark}[On choosing $L$ via equation \eqref{tmin}]
The use of such an upper-bound $L$ conforms with the (by now) well-known restricted isometry property of sub-sampled Gaussian matrices (that is, matrices whose entries are iid Gaussian) in the compressed-sensing literature; for more details, see (for example) \citep[section 1.E]{MR2300700} and the references therein. 
\end{remark}

For $D$ as in \eqref{eq:D1}, this gets us the following bound: 
\begin{eqnarray*}\mathbb P\left[{\mathcal N}_\star
\left(\frac 2DX^\top\boldsymbol{U}\right)> 
{12(84Ae)^{\frac 12}\zeta^{-1}\sigma_{\max}^2}
{C_\sharp}^{\frac 32}(1+\epsilon^{-2}+\epsilon^{-4})
\sqrt{\frac{L\log\left(\frac{ML}\delta\right)}{DM}}\right]&\leq \delta,\end{eqnarray*} In 
view of Theorem \ref{lem:main1}, we correspondingly set \begin{equation}\label{lambda}\lambda=
{24(84Ae)^{\frac 12}\zeta^{-1}\sigma_{\max}^2}
{C_\sharp}^{\frac 32}(1+\epsilon^{-2}+\epsilon^{-4})
\sqrt{\frac{L\log\left(\frac{ML}\delta\right)}{DM}}\end{equation} to force \eqref{eq:16-08-sc2} to be true.



\begin{remark}
From the above, we note that $\lambda$ decreases when any of $M,D$ increases (as the function $x^{-1}\log x\rightarrow 0$ when $x\rightarrow\infty$). This is intuitive since larger $M$ requires that LASSO must not set too many coefficients to 0, and larger the $D$ better the non-regularized estimator is as approximation of the true coefficients.

\end{remark}


Together with Theorem~\ref{lem:2-9-sc4}, the above yields the following bound:
\begin{corollary}[Concrete rates of Prediction Error 
]\label{rates} Suppose that $$\eta=~8(84Ae)^{\frac 12}\zeta^{-1}(1+\epsilon^{-2}+\epsilon^{-4})C_\sharp^{\frac 32}
\sigma_{\max}^2
\sqrt{\frac{L\log\left(\frac
{ML}{\delta}
\right)}
{DM
}}\,,$$ as set in Equation~\eqref{etavalue} above, and \begin{equation*}D\geq \frac{8
\sigma_{\max}^2(1+\epsilon^{-2}+\epsilon^{-4})}{c_0\eta}\log\left(\frac{ML}{\delta}\right)\,.\end{equation*} Then, the following inequality holds with probability at least $1-\delta$: 
    \begin{align*}
        &\mathbb E\left[\frac 1D|\!|\boldsymbol{y}-X\hat{\boldsymbol{\beta}}|\!|_2^2
\mid \mathcal X\right]-\overline{\sigma}^2\\
\leq~&\min_{\alpha\in \mathbb R^{ML}}\left(\frac 1D|\!|X(\boldsymbol{\beta}-
\boldsymbol{\alpha})|\!|_2^2+{\frac{24}{\zeta}(84Ae)^{\frac 12}\sigma_{\max}^2}
{C_\sharp}^{\frac 32}(1+\epsilon^{-2}+\epsilon^{-4})
\sqrt{\frac{L\log\left(\frac{ML}\delta\right)}{DM}}\mathcal N(\boldsymbol{\alpha})\right)\,.
    \end{align*}
\end{corollary}



\subsection{Estimating error of parameters, false discoveries}\label{rip}

We now deliver the treatment of assertions 1. and 2. on estimation error and support (via control of false discoveries) of our Main Theorem 3.1. For this we need to cope with the {\em Restricted Eigenvalue Property} as it typically arises in LASSO analysis (see \cite{MR3357870}).

\subsubsection{Control of the restricted eigenvalue property}
Our prediction guarantees stated so far did not impose restrictions on the minimum number of samples from each component.
For estimation and stability guarantees, however, we need to demand a minimum number of those samples.
In particular, we will need the following proposition, 
on the restricted eigenvalue property of the Gram matrix $X^\top X$. The proof presented in Appendix \ref{sc-07-11-1:proof} follows the same lines of arguments as in \cite{MR3357870}. 
Because of the non-uniform 
nature of the block dimensions of the data matrix in \eqref{data}, as well as 
the individual weights assigned to the components of the coefficient 
$\boldsymbol{\beta}$, we can only hope for a block-wise result as 
stated below.

The following proposition delivers a lower bound on the 
error-expression  $|\!|X_m\boldsymbol{v}_m|\!|_2^2$ with ${\boldsymbol{v}_m}$ the $m-$th component of $\hat{\boldsymbol{\beta}}-\boldsymbol{\beta}$.
\begin{proposition}[Bound on the Restricted Eigenvalue]\label{sc-07-11-1}
For any confidence parameter $\delta>0$ and all vectors $\boldsymbol{v}
=(\boldsymbol{v}_1^\top,\dots,
\boldsymbol{v}_M^\top)^\top\in (\mathbb R^L)^M$, 
the following inequality holds with $\zeta\deq\frac{\epsilon^4}{216}$ 
and $s_m$ even positive integers:

\begin{eqnarray}\label{eq:120}&&\mathbb P\left[\displaystyle\forall~ {m\in[M]}
\inf_{\boldsymbol{v}_m\in\mathbb R^L}
\left(|\!|X_m\boldsymbol{v}_m|\!|_2^2- 
\frac{T_m\sigma_m^2\epsilon^2}
2\left(|\!|\boldsymbol{v}_m|\!|_2^2-\frac{2}
{s_m}|\!|\boldsymbol{v}_m|\!|_1^2\right)\right)\geq 0\right]\nonumber\\
&\geq& 1-2\sum_{m=1}^M\exp\left(-\frac{T_m\min\{\zeta,\zeta^2\}}{2}
+s_m\min\{\log L,\log(21eL/s_m)\}\right)\,.\end{eqnarray}

\end{proposition}

A proof of Proposition~\ref{sc-07-11-1} appears in Appendix \ref{sc-07-11-1:proof}.\\


Setting \begin{equation}\label{aux1}s_m\deq 2\Big
\lfloor\frac {T_m\zeta^2}{8\log L}\Big\rfloor\,,\end{equation} 
we get the following immediate corollary to Proposition~\ref{sc-07-11-1}: 

\begin{corollary}[Bound on the Restricted Eigenvalue for specific $s_m$]\label{res}
If \begin{equation}\label{eq:individual}
T_{\min}\geq 84e\zeta^{-2}
\log L\,,\end{equation} 
where $\zeta= 6^{-3}\epsilon^4$, then the following holds: 
\begin{align}\label{eq:1221}
&\mathbb P\left[\forall {m\in[M]}\inf_{\boldsymbol{v}_m\in\mathbb R^L}
\left(|\!|X_m\boldsymbol{v}_m|\!|_2^2- 
\frac{T_m\sigma_m^2\epsilon^2}
2|\!|\boldsymbol{v}_m|\!|_2^2\left(1-\frac{8L\log L}
{T_m\zeta^2}\right)\right)\geq 0\right]\nonumber\\
\geq~&1-2\sum_{m=1}^M
e^{-\frac{T_m\zeta^2}{4}}\,.\end{align}
\end{corollary}

\begin{proof}
We use $|\!|\boldsymbol{v}_m|\!|_1^2\leq L|\!|\boldsymbol{v}_m|\!|_2^2$ in Proposition~\ref{sc-07-11-1}.
\end{proof}

\bigskip

\subsubsection{Bounds on estimation error of the 
autoregressive coefficients, and false discovery
}\label{rip1}
This section contains the bound on the error of estimation of the autoregressive coefficients, and the false discovery.

The following is a compact notation used in the proof of the theorem below (also in the proof of Theorem \ref{lem:2-9-sc4}).

\begin{definition}[Notation]
\begin{equation}\label{nlew}{\mathcal N}_{\leq L_0}(\boldsymbol{v})
\deq\sum_{\ell=1}^{L_0}\sqrt{|{\mathcal G}_\ell|}\cdot
|\!|\boldsymbol{v}|\!|_{{\mathcal G}_\ell}
\,.\end{equation}
\end{definition}
The theorem below bounds the $\ell^2$-error in estimating the autoregressive coefficients $\boldsymbol{\beta}$, as well as the size of the set of false positives, using the penalized LASSO formulation as in \eqref{eq:1}.
\begin{theorem}[Bounds on the Estimation Error and False Discovery]\label{prop:21-10-sc1}
Let 
$\hat{\boldsymbol{\beta}}$ be 
the solution to the group-regularized LASSO in \eqref{eq:1} with $\lambda$ as in Equation \eqref{lambda} above --- namely, 
\begin{equation*}\label{lambda00}\lambda={24(84Ae)^{\frac 12}\zeta^{-1}\sigma_{\max}^2}
{C_\sharp}^{\frac 32}(1+\epsilon^{-2}+\epsilon^{-4})
\sqrt{\frac{L\log\left(\frac{ML}\delta\right)}{DM}}\end{equation*} where $\zeta=
6^{-3}\epsilon^4$; if \begin{equation}\label{alpha}\alpha\deq\min_{m\in[M]}
\frac{T_m\sigma_m^2}D\,,\end{equation} and 
$$D\geq \frac{8
\sigma_{\max}^2(1+\epsilon^{-2}+\epsilon^{-4})}{c_0\eta}\log\left(\frac{ML}{\delta}\right)\,,$$
then, for any confidence parameter $\delta>0$, 
the inequality \begin{equation}\label{vvv}|\!|\hat{\boldsymbol{\beta}}-\boldsymbol{\beta}|\!|_2\leq 
\frac{81(84Ae)^{\frac 12}LL_0\sigma_{\max}^2C_{\sharp}^{\frac 32}(1+\epsilon^{-2}+\epsilon^{-4})}{\zeta\alpha\epsilon^2}\sqrt{\frac{\log\left(\frac{ML}{\delta}\right)}D}\end{equation} holds with 
probability at least \begin{equation}\label{prob1}(1-\delta)\left(1-2\sum_{m=1}^ML^{-\frac{T_m}{4\log L}}\right)\,,\end{equation} 
provided $T_{\min}\geq 84eA\zeta^{-2}L_0\log L$ with $\zeta=
6^{-3}\epsilon^4$.\\

Moreover, introducing the notation $S_\lambda \deq \{j:|\hat{\boldsymbol{\beta}}_j|>\lambda\}$, then the false discovery is bounded by the following inequality holding with probability as in \eqref{prob1} above: \begin{equation}\label{fdsize}|S_\lambda\setminus \operatorname{supp}({\boldsymbol{\beta}})|~\leq \frac{243(84Ae)^{\frac 12}LL_0^{\frac 32}\sigma_{\max}^2C_{\sharp}^{\frac 32}(1+\epsilon^{-2}+\epsilon^{-4})}{\zeta\alpha\epsilon^2\lambda}\sqrt{\frac{\log\left(\frac{ML}{\delta}\right)}D}\,.\end{equation}
\end{theorem}

\begin{remark}
In the above, the term inside the square root decreases asymptotically, and the terms outside may change or stay fixed (depending on how the number of samples for each component increases).
\end{remark}

\begin{proof}[Proof of Theorem \ref{prop:21-10-sc1}]
For ease of notation, write $$\boldsymbol{v}\deq\hat{\boldsymbol{\beta}}-\boldsymbol{\beta}\,.$$ 
By Proposition~\ref{sc-07-11-1} above (which can be 
applied since $T_{\min}\geq 84e\zeta^{-2}\log L$), one has \begin{eqnarray*}\frac 1D|\!|X\boldsymbol{v}
|\!|_2^2 &\geq& \frac 1D\sum_{m=1}^M|\!|X_m\boldsymbol{v}_m|\!|_2^2\\
&\geq& \frac{\epsilon^2}2\sum_{m=1}^M\frac{\sigma_m^2T_m}
{D}|\!|\boldsymbol{v}_m|\!|_2^2\left(1-\frac{8L\log L}
{T_m\zeta^2}\right)\\
&>& \frac{4\alpha\epsilon^2}{9}|\!|\boldsymbol{v}|\!|_2^2\,. 
\end{eqnarray*}
To this, we now apply (see \eqref{eq:18-10-sc401} in Appendix \ref{lem:2-9-sc4:proof}) the inequality 
\begin{equation*}
\frac 1D|\!|X(\boldsymbol{\beta}-\hat{\boldsymbol{\beta}})|\!|_2^2 
\leq 3\lambda{\mathcal N}_{\leq L_0}(\boldsymbol{\beta}-\boldsymbol{\hat\beta})\,,\end{equation*} 
to obtain
\begin{eqnarray*}
|\!|\boldsymbol{v}|\!|_2^2&\leq & \frac{27\lambda}{8\alpha\epsilon^2}
\mathcal N_{\leq L_0}(\boldsymbol{v})\\
&\leq&\frac{27\lambda L_0\sqrt{ML}}{8\alpha\epsilon^2}
|\!|\boldsymbol{v}|\!|_2\\
\Rightarrow\hspace{1cm}|\!|\boldsymbol{v}|\!|_2&\leq &
\frac{27\lambda L_0\sqrt{ML}}{8\alpha\epsilon^2}
\end{eqnarray*}
With the choice of $\lambda$ as in equation \eqref{lambda}, namely, $$\lambda={24(84Ae)^{\frac 12}\zeta^{-1}\sigma_{\max}^2}
{C_\sharp}^{\frac 32}(1+\epsilon^{-2}+\epsilon^{-4})
\sqrt{\frac{L\log\left(\frac{ML}\delta\right)}{DM}}\,,$$ it now 
follows that \begin{eqnarray*}|\!|\boldsymbol{v}|\!|_2&\leq &
\frac{81(84Ae)^{\frac 12}LL_0\sigma_{\max}^2C_{\sharp}^{\frac 32}(1+\epsilon^{-2}+\epsilon^{-4})}{\zeta\alpha\epsilon^2}\sqrt{\frac{\log\left(\frac{ML}{\delta}\right)}D}\,,
\end{eqnarray*} as claimed. To get the bound on the false discovery, write $S\deq\operatorname{supp}({\boldsymbol{\beta}})$; the bound on the false discovery can be obtained as follows: \begin{align*}|S_\lambda\setminus S|&= \sum_{j\in [ML]\setminus S}1_{|\hat{\boldsymbol{\beta}}_j|>\lambda}(\hat{\boldsymbol{\beta}})\\ 
&= \sum_{j\in [ML]\setminus S}1_{|\hat{\boldsymbol{v}}_j|>\lambda}(\hat{\boldsymbol{v}})\\ &\leq \frac 1\lambda\sum_{j\in [ML]\setminus S}|\hat{\boldsymbol{v}}_j|\\ \mbox{by remark \eqref{rem:bounds}}\hspace{1cm}&\leq \frac 3\lambda\sum_{j\in S}|\hat{\boldsymbol{v}}_j|\\ 
&\leq \frac {3\sqrt{L_0}}\lambda|\!|\hat{\boldsymbol{v}}|\!|_2\end{align*} which yields the bound in \eqref{fdsize} by inequality \eqref{vvv}.
\end{proof}


\begin{remark}
For the explicit choice of $\lambda$ as in equation \eqref{lambda}, namely, 
\begin{equation*}\lambda={24(84Ae)^{\frac 12}\zeta^{-1}\sigma_{\max}^2}
{C_\sharp}^{\frac 32}(1+\epsilon^{-2}+\epsilon^{-4})
\sqrt{\frac{L\log\left(\frac{ML}\delta\right)}{DM}}\end{equation*} with $\zeta =6^{-3}\epsilon^4$, 
the inequality \eqref{fdsize} becomes \begin{equation}\label{fdsize1}|S_\lambda\setminus \operatorname{supp}({\boldsymbol{\beta}})|~\leq \frac{81(ML)^{\frac 12}L_0^{\frac 32}}{8\alpha\epsilon^2}\,.\end{equation}
\end{remark}



\subsection{\bf{Stability of the estimated AR model}}\label{stasta}
We finally prove that the coefficients estimated as per the overlapping-group-LASSO in 
\eqref{eq:1}, with $\lambda$ as in \eqref{lambda}, lie in the region for stability of 
univariate lag-$L$ autoregressive processes, even when the number of post-samples 
is 'not too large'. More specifically, we have the following 
theorem. 
\begin{theorem}[Stability Guarantee]\label{sc1-07-11:2}
Let $A\geq 1$ be a confidence parameter. Let $\hat{\boldsymbol{\beta}}$ 
be the output of the group-LASSO in \eqref{eq:1} above, using $D$ post-samples, where  
\begin{equation}\label{D-bound}D\geq 
3^9\cdot (84Ae)C_\sharp^5(1+\epsilon^{-2}+\epsilon^{-4})^2\left(\frac{\sigma_{\max}}{\sigma_{\min}}\right)^4(\epsilon^3\zeta)^{-2}M^aL_0^2L^3
\log\left(\frac{ML}\delta\right)\log (2L)\ .\end{equation}
If $T_{\min}\geq 84eA\zeta^{-2}L_0\log L$, and $L$ satisfies \begin{equation}\label{L-bound}
    n_{\min}=L+84Ae\zeta^{-2}L\log\left(\frac{ML}\delta\right)\,,
\end{equation}
with $\zeta=6^{-3}\epsilon^4$, then the AR-models --- fitted with coefficients $\hat{\boldsymbol{\beta}}_0$ returned by Algorithm~\ref{alg:learnar} ("AR Coefficient Estimation Pipeline") --- are stable, with probability as in \eqref{prob1} --- in the following scenarios: 
\begin{enumerate}
\item when all the time series are different realizations of a unique 
underlying stochastic process, $a\geq1$, and $$\hat{\boldsymbol{\beta}}_0\deq
\frac 1M\sum_{m=1}^M\hat{\boldsymbol{\beta}}_m\,;$$ 
\item when all the time series are realizations of 
different underlying stochastic processes (equivalently, all the 
$\boldsymbol{\beta}_m$'s are different), $a\geq 2$, and $\hat{\boldsymbol{\beta}}_0\deq
\hat{\boldsymbol{\beta}}$.
\end{enumerate} 
\end{theorem}

The lower bound on $T_{\min}$ aligns with the upper bound in Equation~\eqref{tmin} as $L\geq L_0$. The stated value of $n_{\min}$ chooses the smallest value of~$L$ to satisfy both the upper and lower bound on~$T_{\min}$.

\begin{proof}[Proof of Theorem \ref{sc1-07-11:2}]
\begin{enumerate}
\item We start with the first case. The idea of the proof is 
as follows. All the component $\hat{\boldsymbol{\beta}}_m$'s 
of $\hat{\boldsymbol{\beta}}$ are approximations of the same underlying 
${\boldsymbol{\beta}}_0\deq{\boldsymbol{\beta}}_m$ for all $m\in [M]$. Therefore, by the bound in Proposition \ref{prop:21-10-sc1} 
above and by convexity of the square function, their mean must be 
$\ell^2$-close to $\boldsymbol{\beta}_0$, 
with $D=O(ML_0^2L^3\log(ML)\log (2L))$ many samples. Moreover, the $0.5L^{-1}\epsilon$ 
perturbation of the coefficients preserves stability of the $\epsilon$-stable process.

More explicitly, we have 
\begin{align*}\Bigg|\!\Bigg|\left(\sum_{m=1}^M\frac 1M
\hat{\boldsymbol{\beta}}_m\right)-{\boldsymbol{\beta}}_0
\Bigg|\!\Bigg|_2^2&=\Bigg|\!\Bigg|\frac 1M\sum_{m=1}^M
(\hat{\boldsymbol{\beta}}_m-{\boldsymbol{\beta}}_m)
\Bigg|\!\Bigg|_2^2\\ \mbox{convexity}~\Rightarrow\hspace{1cm}&\leq\frac 1M\sum_{m=1}^M|\!|\hat{\boldsymbol{\beta}}_m-{\boldsymbol{\beta}}_m|\!|_2^2\\
&=\frac 1M|\!|\hat{\boldsymbol{\beta}}-{\boldsymbol{\beta}}|\!|_2^2,\end{align*} 
and by Theorem \ref{prop:21-10-sc1}, this yields 
\begin{align}\Bigg|\!\Bigg|\left(\sum_{m=1}^M\frac 1M
\hat{\boldsymbol{\beta}}_m\right)-{\boldsymbol{\beta}}
\Bigg|\!\Bigg|_2&\leq \frac 1{\sqrt M}
|\!|\hat{\boldsymbol{\beta}}-{\boldsymbol{\beta}}|\!|_2\nonumber\\
&\leq \frac{81(84Ae)^{\frac 12}LL_0\sigma_{\max}^2C_{\sharp}^{\frac 32}(1+\epsilon^{-2}+\epsilon^{-4})}{\zeta\alpha\epsilon^2}\sqrt{\frac{\log\left(\frac{ML}{\delta}\right)}{MD}}
\end{align} with high probability. 
Note that \begin{equation}\label{alpha2}\alpha=\min_{m\in[M]}
\frac{T_m\sigma_m^2}D\geq \frac{T_{\min}\sigma_{\min}^2}D
\geq \frac{T_{\min}\sigma_{\min}^2}{MT_{\max}}
\geq \frac{\sigma_{\min}^2}{MC_{\sharp}}\,.\end{equation} When $$D\geq
3^9\cdot (84Ae)C_\sharp^5(1+\epsilon^{-2}+\epsilon^{-4})^2\left(\frac{\sigma_{\max}}{\sigma_{\min}}\right)^4(\epsilon^3\zeta)^{-2}ML_0^2L^3
\log\left(\frac{ML}\delta\right)\log (2L)\,,$$ 
this yields \begin{align*}\Bigg|\!\Bigg|\left(\sum_{m=1}^M\frac 1M
\hat{\boldsymbol{\beta}}_m\right)-{\boldsymbol{\beta}}
\Bigg|\!\Bigg|_2&\leq \frac{81(84Ae)^{\frac 12}LL_0\sigma_{\max}^2C_{\sharp}^{\frac 32}(1+\epsilon^{-2}+\epsilon^{-4})}{\zeta\alpha\epsilon^2}\sqrt{\frac{\log\left(\frac{ML}{\delta}\right)}{MD}}\\
&\leq \frac{\epsilon\sigma_{\min}^2}{\sqrt 3 C_{\sharp}M\alpha}\sqrt{\frac{1}{L\log (2L)}}\\
\eqref{alpha2}~\Rightarrow\hspace{1cm}&<
\frac {\epsilon}{\sqrt 3
}\sqrt{\frac1{L\log (2L)}}\,.\end{align*}
Since $\epsilon\in (0,1)$, the stability follows by the triangle inequality. More explicitly, 
if $\hat{\boldsymbol{\beta}}_0\deq M^{-1}(\hat{\boldsymbol{\beta}}_1+\cdots+
\hat{\boldsymbol{\beta}}_M)$, then for 
every $z\in \mathbb D$, we have : 
\begin{align*}|\boldsymbol{f}_{\hat
{\boldsymbol{\beta}}_0}(z)|&= |1-\hat
{\boldsymbol{\beta}}_0\cdot (z,z^2,\cdots,z^L)|\\ &= |1-
{\boldsymbol{\beta}}_0\cdot (z,z^2,\cdots,z^L)-(\hat
{\boldsymbol{\beta}}_0-{\boldsymbol{\beta}}_0)\cdot (z,z^2,\cdots,z^L)|\\ 
&\geq |\boldsymbol{f}_{{\boldsymbol{\beta}_0}}(z)|-|\!|\hat
{\boldsymbol{\beta}}_0-{\boldsymbol{\beta}}_0|\!|_2 \cdot |\!|(z,z^2,\cdots,z^L)|\!|_2\
\\ &> \epsilon-\frac {\epsilon}{\sqrt 3
}\sqrt{\frac1{\log (2L)}}\\ &>0\,.\end{align*} Evidently, 
this shows that the autoregressive process with ${\hat
{\boldsymbol{\beta}}_0}$ coefficients is stable.\\

\item The arguments are similar in the second case. Since 
$$D\geq 3^9\cdot (84Ae)C_\sharp^5(1+\epsilon^{-2}+\epsilon^{-4})^2\left(\frac{\sigma_{\max}}{\sigma_{\min}}\right)^4(\epsilon^3\zeta)^{-2}M^2L_0^2L^3
\log\left(\frac{ML}\delta\right)\log (2L)\,,$$ by 
Theorem \ref{prop:21-10-sc1}, for any $m\in [M]$ we have 
\begin{align*}|\!|\hat{\boldsymbol{\beta}}_m-
{\boldsymbol{\beta}}_m|\!|_2&\leq |\!|\hat{\boldsymbol{\beta}}-
{\boldsymbol{\beta}}|\!|_2\nonumber\\ &\leq 
\frac{81(84Ae)^{\frac 12}LL_0\sigma_{\max}^2C_{\sharp}^{\frac 32}(1+\epsilon^{-2}+\epsilon^{-4})}{\zeta\alpha\epsilon^2}\sqrt{\frac{\log\left(\frac{ML}{\delta}\right)}{D}}\nonumber\\
\eqref{alpha2}~\Rightarrow\hspace{1cm}&<
\frac {\epsilon}{\sqrt 3
}\sqrt{\frac1{L\log (2L)}}\,.
\end{align*} As in the first case, this implies stability for each of the components.

\end{enumerate}
\end{proof}

\section{Algorithmic Aspects}

In this section we briefly introduce the main gazette of our algorithm --- namely, the proximal operator ---  to understand the procedure of solving LASSO as was done in \cite{MR4209452}. To sketch the outline of the standard procedure for solving regularized LASSO penalized with an overlapping group-norm, we start with the following definition (see \cite{MR2845676}, for example).

\begin{definition}[Proximal operator]
    Given a norm $\mathcal N$ on $\mathbb R^{ML}$, and a tuning parameter $\lambda$, 
    the associated proximal operator $\operatornamewithlimits{Prox}_{\mathcal N,\lambda}$ is defined for every $\boldsymbol{\alpha}\in\mathbb R^{ML}$ as the optimum value of the following convex problem: \begin{align*}
        \operatornamewithlimits{Prox}_{\mathcal N,\lambda}(\boldsymbol{\alpha})\deq  \operatornamewithlimits{argmin}_{\boldsymbol{\beta}}\left\{\frac 12|\!|\boldsymbol{\beta}-\boldsymbol{\alpha}|\!|_2^2+\lambda\mathcal N(\boldsymbol{\beta})\right\}\,.
    \end{align*}
\end{definition}

By \cite{MR2549566}, the proximal operator $\operatornamewithlimits{Prox}_{\mathcal N,\lambda}$ --- for $\mathcal N$ as defined in \eqref{Neq} --- is the composition $$\operatornamewithlimits{Prox}_{\mathcal G_1,\lambda}\circ\cdots\circ \operatornamewithlimits{Prox}_{\mathcal G_L,\lambda}$$ of the proximal operators for the individual groups 
and can be computed inductively --- starting from $\operatornamewithlimits{Prox}_{\mathcal G_L,\lambda}$, which is the well-known soft-thresholding --- in $O(ML)$ computational steps. By \citet[Proposition 3.1(iii)b]{MR2203849}, the solutions to the regularized LASSO problem in \eqref{eq:1} are precisely the fixed points of the operator \begin{equation}\label{proximal-fixed}\boldsymbol{\beta}\longmapsto \operatornamewithlimits{Prox}_{\mathcal N,\lambda}\left(\boldsymbol{\beta}+\frac {2\lambda}D\boldsymbol{X}^\top(\boldsymbol{y}-\boldsymbol{X}\boldsymbol{\beta})\right)\,.\end{equation} Here we use an appropriate $\lambda$ (as  given in \eqref{lambda} above). As \cite{MR4209452} observed, the proximal operator can be evaluated via duality. The proximal gradient method of \cite{MR2845676} then finds the fixed point (which exists and is unique by convexity of \eqref{eq:1}) of this proximal operator. For pseudo-code of this procedure, see \cite{MR4209452}, where an accelerated version of the proximal descent method was employed for achieving quadratic convergence rate.



\section{Conclusion}
We have established a set-up (see Equation \eqref{eq:1} and the discussion preceding this equation) in which LASSO --- regularized with a hierarchical group norm --- can be used to derive statistical guarantees in terms of the one-step ahead prediction error (Theorem \ref{lem:2-9-sc4}) in the realms of multiple $\epsilon$-stable (Definition \eqref{defn:new1}) univariate autoregressive processes of different lengths but identical true lag $L_0$. 
The results presented here assume no prior knowledge of the true lag (or any upper-bound of the true lag); in fact, we show that the sample size itself suggests a certain lag $\hat{L}$ to be used for the group-LASSO, and given an appropriately large sample size, such that $\hat{L}$ will be an upper-bound of the true lag $L_0$. Moreover, this $\hat{L}$ will be of the order that is required for our theoretical guarantees to hold. We proved that the group-LASSO formulated with a suitable tuning parameter $\lambda$ estimates the AR coefficients with an arbitrarily high degree of accuracy. We also showed the support of the estimated coefficient-set approximately matches the support of the original parameters (Theorem \ref{prop:21-10-sc1}). Finally, we proved that the fitted models with coefficients as estimated by the group-LASSO are $\epsilon$-stable (Theorem \ref{sc1-07-11:2}), a property that is known in the literature solely for Yule-Walker estimates of the parameters of univariate autoregressive processes.

From a theoretical perspective, it will be interesting to investigate adaptations of the group-LASSO method to the case of multiple decoupled AR processes with multivariate components. We expect that this will require, among others, (1) additional techniques to deal with the group-norm, and (2) integrating the stability issues and the restricted eigenvalue issues; these will be technically far more demanding in the multivariate components settings. On the practical front, the most important question is to get a better hold on the tuning parameter $\lambda$ (see Equation \eqref{lambda}), which requires better constant/parameters than, for example, those appearing in the proof of Lemma \ref{lemma4}. A better control on these will enable a more realistic estimate of $\hat{L}$ --- to be used by the group-LASSO as an upper-bound on the true lag $L_0$.


\begin{funding}
S.C.~and J.L.~acknowledge funding from the Deutsche Forschungsgemeinschaft (DFG) under grant number 451920280. S.C.~was partially supported by a `Research Assistantship Fellowship for Early Postdocs' from Ruhr-Universit\"at Bochum Research School by means of the German Academic Exchange Service (DAAD) STIBET funds. 
\end{funding}


\bigskip

\appendix



\begin{appendix}

\section{Proof of Theorem \ref{lem:2-9-sc4}}\label{lem:2-9-sc4:proof}
\begin{proof}
Denote the data by ${\mathcal X}$. Define the random vector $\boldsymbol{U}\in\mathbb R^D$ as follows: $$\boldsymbol{U}\deq (U_1^1,\dots,U_{T_1}^1,U_1^2,\dots,U_{T_2}^2,\cdots,U_1^M,
\dots,U_{T_M}^M)^\top\,.$$ We note the following series of 
inequalities, all of which follow from linearity of expectation: here, we 
write $\boldsymbol{\beta}$ for the true coefficient vector of the model. One has \begin{eqnarray*}&&
\mathbb E\left[|\!|\boldsymbol{y}-X\hat{\boldsymbol{\beta}}|\!|_2^2
\mid X\right]\\ &=&\mathbb E\left[|\!|\boldsymbol{y}-
X{\boldsymbol{\beta}}|\!|_2^2\mid X\right]+
\mathbb E\left[|\!|X(\boldsymbol{\beta}-\hat{\boldsymbol{\beta}})|\!|_2^2
\mid X\right]\\ && \hspace{3cm} +2\mathbb E
\left[\langle \boldsymbol{y}-X\boldsymbol{\beta}, X\boldsymbol{\beta}
-X\hat{\boldsymbol{\beta}}\rangle\mid X\right]\\ &
=&\mathbb E\left[|\!|\boldsymbol{U}|\!|_2^2\mid X\right]+
\mathbb E\left[|\!|X(\boldsymbol{\beta}-\hat{\boldsymbol{\beta}})|\!|_2^2
\mid X\right]\\ && \hspace{3cm} +2\mathbb E
\left[\langle\boldsymbol{U}, X\boldsymbol{\beta}-X\hat{\boldsymbol{\beta}}
\rangle\mid X\right]\\ &=&T_1\sigma_1^2+\cdots+
T_M\sigma_M^2+\mathbb E\left[|\!|X(\boldsymbol{\beta}-
\hat{\boldsymbol{\beta}})|\!|_2^2\mid X\right]+
2\langle\mathbb E\left[\boldsymbol{U}\mid X\right], 
X\boldsymbol{\beta}-X\hat{\boldsymbol{\beta}}\rangle\\ &=& T_1
\sigma_1^2+\cdots+T_M\sigma_M^2+|\!|X(\boldsymbol{\beta}-
\hat{\boldsymbol{\beta}})|\!|_2^2\,,\end{eqnarray*} where the last equality 
follows from $\mathbb E\left[\boldsymbol{U}\mid X\right]=
\boldsymbol{U}
$, and $\boldsymbol U\perp (X\boldsymbol{\beta}-X\hat{\boldsymbol{\beta}})$. Therefore, we 
need to derive an estimate of the error $|\!|X(\boldsymbol{\beta}-
\hat{\boldsymbol{\beta}})|\!|_2^2$. We notice that \begin{eqnarray*}|\!|X
(\boldsymbol{\beta}-\hat{\boldsymbol{\beta}})|\!|_2^2&=&\sum_{m=1}^M
\sum_{t=1}^{T_m}\left(\sum_{l=1}^L\left(\hat\beta_l^m-\boldsymbol{\beta}_l^m
\right)x_{t-l}^m\right)^2\,,\end{eqnarray*} where $\{\boldsymbol{\beta}_l^m\}$ 
denotes the true-parameters of the models. Using the fact that $\hat{\boldsymbol{\beta}}$ is a minimizer of $\frac 1D|\!|\boldsymbol{y}-X\boldsymbol{\alpha}|\!|_2^2+
\lambda\sum_{l=1}^L
\mathcal N({\boldsymbol{\alpha}})$, over $\boldsymbol{\alpha} 
\in \mathbb R^{ML}$, one derives \begin{eqnarray*}
&\frac 1D|\!|\boldsymbol{y}-X\hat{\boldsymbol{\beta}}|\!|_2^2+
\lambda\mathcal N(\hat{\boldsymbol{\beta}})&
\leq \frac 1D |\!|\boldsymbol{y}-X\boldsymbol{\alpha}|\!|_2^2+
\lambda\mathcal N({\boldsymbol{\alpha}})\\ 
\Rightarrow & \frac 1D |\!|X(\boldsymbol{\beta}-\hat{\boldsymbol{\beta}})|\!|_2^2&\leq 
\frac 1D|\!|X(\boldsymbol{\beta}-\boldsymbol{\alpha})|\!|_2^2+
\frac 2D \langle \boldsymbol{U}, X(\hat{\boldsymbol{\beta}}-\boldsymbol{\alpha})\rangle
+\lambda\mathcal N(\boldsymbol{\alpha})
-\lambda\mathcal N(\hat{\boldsymbol{\beta}})\\ & &\leq \frac 1D|\!|X(\boldsymbol{\beta}-
\boldsymbol{\alpha})|\!|_2^2+\frac 2D \langle 
X^\top\boldsymbol{U}, \hat{\boldsymbol{\beta}}-\boldsymbol{\alpha}
\rangle+\lambda\mathcal N(\boldsymbol{\alpha})
-\lambda\mathcal N(\hat{\boldsymbol{\beta}})\end{eqnarray*} for 
all $\boldsymbol{\alpha} \in \mathbb R^{ML}$. By the Cauchy-Schwarz inequality, 
this implies \begin{eqnarray*}
\frac 1D|\!|X(\boldsymbol{\beta}-\hat{\boldsymbol{\beta}})|\!|_2^2 
&\leq & \frac 1D|\!|X(\boldsymbol{\beta}-
\boldsymbol{\alpha})|\!|_2^2+\frac 2D \mathcal N_\star(X^\top
\boldsymbol{U})\mathcal N (\hat{\boldsymbol{\beta}}-\boldsymbol{\alpha})
+\lambda(\mathcal N ({\boldsymbol{\alpha}})
-\mathcal N (\hat{\boldsymbol{\beta}}))\\ & \leq & \frac 1D|\!|X(\boldsymbol{\beta}-
\boldsymbol{\alpha})|\!|_2^2+\lambda\mathcal N (\hat{\boldsymbol{\beta}}
-{\boldsymbol{\alpha}})+\lambda(\mathcal N ({\boldsymbol{\alpha}})-
\mathcal N (\hat{\boldsymbol{\beta}}))\end{eqnarray*} since $\lambda\geq 
\frac 2D\mathcal N_\star(X^\top\boldsymbol{U})$. Finally, 
the triangle inequality $\mathcal N (\hat{\boldsymbol{\beta}}
-{\boldsymbol{\alpha}})\leq \mathcal N (\hat{\boldsymbol{\beta}})
+\mathcal N ({\boldsymbol{\alpha}})$ establishes the 
first part of the theorem.

In the following, we will choose a slightly different threshold for the tuning 
parameter $\lambda$, as proposed in the statement of 
theorem. Let $\boldsymbol{v}\deq \boldsymbol{\beta}+
\hat{\boldsymbol{\beta}}$, and notice that the true lag $L_0$ satisfies
$$L_0=
\max\{\ell\in[L]:|\!|\boldsymbol{\beta}|\!|_{{\mathcal G}_\ell}
\neq 0\}\,.$$ We have 
\begin{align}\mathcal N ({\boldsymbol{\beta}})
-\mathcal N (\hat{\boldsymbol{\beta}})&=
\mathcal N ({\boldsymbol{\beta}})
-\mathcal N (\boldsymbol{\beta}+\boldsymbol{v})\nonumber\\
&=\sum_{\ell=1}^{L_0}\sqrt{|{\mathcal G}_\ell|}\cdot
|\!|\boldsymbol{\beta}|\!|_{{\mathcal G}_\ell}+\sum_{\ell=L_0+1}^L\sqrt{|{\mathcal G}_\ell|}\cdot
|\!|\cancelto{0}{\boldsymbol{\beta}}|\!|_{{\mathcal G}_\ell}\nonumber\\ 
&\hspace{1cm} -\sum_{\ell=1}^{L_0}\sqrt{|{\mathcal G}_\ell|}\cdot
|\!|\boldsymbol{\beta}+\boldsymbol{v}|\!|_{{\mathcal G}_\ell}-\sum_{\ell=L_0+1}^L\sqrt{|{\mathcal G}_\ell|}\cdot
|\!|\cancelto{0}{\boldsymbol{\beta}}+\boldsymbol{v}|\!|_{{\mathcal G}_\ell}\nonumber\\
&=\sum_{\ell=1}^{L_0}\sqrt{|{\mathcal G}_\ell|}\cdot
|\!|\boldsymbol{\beta}|\!|_{{\mathcal G}_\ell}-\sum_{\ell=1}^{L_0}\sqrt{|{\mathcal G}_\ell|}\cdot
|\!|\boldsymbol{\beta}+\boldsymbol{v}|\!|_{{\mathcal G}_\ell}-\sum_{\ell=L_0+1}^L\sqrt{|{\mathcal G}_\ell|}\cdot
|\!|\boldsymbol{v}|\!|_{{\mathcal G}_\ell}\nonumber\\
&\leq \sum_{\ell=1}^{L_0}\sqrt{|{\mathcal G}_\ell|}\cdot
|\!|\boldsymbol{v}|\!|_{{\mathcal G}_\ell}-\sum_{\ell=L_0+1}^L\sqrt{|{\mathcal G}_\ell|}\cdot
|\!|\boldsymbol{v}|\!|_{{\mathcal G}_\ell},
\end{align} 
where the last step is due to the triangle 
inequality: $|\!|\boldsymbol{\beta}+\boldsymbol{v}|\!|_{{\mathcal G}_\ell}
\geq |\!|\boldsymbol{\beta}|\!|_{{\mathcal G}_\ell}-
|\!|\boldsymbol{v}|\!|_{{\mathcal G}_\ell}$. Suppose that $\lambda\geq \frac 4D\mathcal N_\star(X^\top\boldsymbol{U})$. 
Using the fact that $\hat{\boldsymbol{\beta}}$ is a 
minimizer of the objective $\frac 1D|\!|\boldsymbol{y}-X\boldsymbol{\alpha}|\!|_2^2+
\lambda\mathcal N({\boldsymbol{\alpha}})
$, over $\boldsymbol{\alpha} 
\in \mathbb R^{ML}$, one derives \begin{eqnarray*}
&\frac 1D|\!|\boldsymbol{y}-X\hat{\boldsymbol{\beta}}|\!|_2^2+
\lambda\mathcal N(\hat{\boldsymbol{\beta}})&
\leq \frac 1D|\!|\boldsymbol{y}-X\boldsymbol{\beta}|\!|_2^2+
\lambda\mathcal N({\boldsymbol{\beta}})\\ 
\Rightarrow & \frac 1D|\!|X(\boldsymbol{\beta}-\hat{\boldsymbol{\beta}})|\!|_2^2&\leq 
\frac 2D \langle X^\top\boldsymbol{U}, \hat{\boldsymbol{\beta}}-\boldsymbol{\beta}
\rangle+\lambda\mathcal N(\boldsymbol{\beta})
-\lambda\mathcal N(\hat{\boldsymbol{\beta}})\,.\end{eqnarray*} By the
Cauchy-Schwarz inequality, 
this implies \begin{align}\label{eq:16-08-sc1}
\frac 1D|\!|X(\boldsymbol{\beta}-\hat{\boldsymbol{\beta}})|\!|_2^2 &
\leq \frac 2D \mathcal N_\star(X^\top
\boldsymbol{U})\mathcal N (\hat{\boldsymbol{\beta}}-\boldsymbol{\beta})
+\lambda(\mathcal N ({\boldsymbol{\beta}})
-\mathcal N (\hat{\boldsymbol{\beta}}))\nonumber\\ & \leq \frac \lambda2
\mathcal N (\hat{\boldsymbol{\beta}}
-{\boldsymbol{\beta}})+\lambda(\mathcal N ({\boldsymbol{\beta}})-
\mathcal N (\hat{\boldsymbol{\beta}}))\hspace{1cm}
\because~ \lambda\geq 
\frac 4D\mathcal N_\star(X^\top\boldsymbol{U})\nonumber\\ 
&\leq \frac{\lambda}2\min\left\{3\mathcal N(\boldsymbol{\beta}
-\hat{\boldsymbol{\beta}}), 3\mathcal N(\boldsymbol{\beta})
-\mathcal N(\hat{\boldsymbol{\beta}})\right\}
\,.\end{align}
The inequalities 
$$0\leq \frac 2{\lambda D}|\!|X(\boldsymbol{\beta}-\hat{\boldsymbol{\beta}})|\!|_2^2
\leq \mathcal N (\hat{\boldsymbol{\beta}}
-{\boldsymbol{\beta}})+2(\mathcal N ({\boldsymbol{\beta}})-
\mathcal N (\hat{\boldsymbol{\beta}}))\,,$$ and 
$$\mathcal N ({\boldsymbol{\beta}})
-\mathcal N (\hat{\boldsymbol{\beta}})\leq \sum_{\ell=1}^{L_0}\sqrt{|{\mathcal G}_\ell|}\cdot
|\!|\boldsymbol{v}|\!|_{{\mathcal G}_\ell}-\sum_{\ell=L_0+1}^L\sqrt{|{\mathcal G}_\ell|}\cdot
|\!|\boldsymbol{v}|\!|_{{\mathcal G}_\ell}$$
yield the following: 
\begin{align*}
0&\leq  \mathcal N (\hat{\boldsymbol{\beta}}
-{\boldsymbol{\beta}})+2(\mathcal N ({\boldsymbol{\beta}})-
\mathcal N (\hat{\boldsymbol{\beta}}))\\
&\leq\mathcal N (\hat{\boldsymbol{v}})+
2\sum_{\ell=1}^{L_0}\sqrt{|{\mathcal G}_\ell|}\cdot
|\!|\boldsymbol{v}|\!|_{{\mathcal G}_\ell}-2\sum_{\ell=L_0+1}^L\sqrt{|{\mathcal G}_\ell|}\cdot
|\!|\boldsymbol{v}|\!|_{{\mathcal G}_\ell}\\
&=3\sum_{\ell=1}^{L_0}\sqrt{|{\mathcal G}_\ell|}\cdot
|\!|\boldsymbol{v}|\!|_{{\mathcal G}_\ell}-\sum_{\ell=L_0+1}^L\sqrt{|{\mathcal G}_\ell|}\cdot
|\!|\boldsymbol{v}|\!|_{{\mathcal G}_\ell}
\,.\end{align*} That is, we have
\begin{align}\label{eq:18-10-sc3}
\sum_{\ell=L_0+1}^L\sqrt{|{\mathcal G}_\ell|}\cdot
|\!|\boldsymbol{\beta}-\boldsymbol{\hat\beta}|\!|_{{\mathcal G}_\ell}&\leq 
3\sum_{\ell=1}^{L_0}\sqrt{|{\mathcal G}_\ell|}\cdot
|\!|\boldsymbol{\beta}-\boldsymbol{\hat\beta}|\!|_{{\mathcal G}_\ell}
\,.\end{align} By equation \eqref{eq:16-08-sc1}, we then have 
\begin{align}\label{eq:18-10-sc401}
\frac 1D|\!|X(\boldsymbol{\beta}-\hat{\boldsymbol{\beta}})|\!|_2^2 &
\leq 2\lambda\sum_{\ell=1}^{L_0}\sqrt{|{\mathcal G}_\ell|}\cdot
|\!|\boldsymbol{\beta}-\boldsymbol{\hat\beta}|\!|_{{\mathcal G}_\ell}\\
&=2\lambda{\mathcal N}_{\leq L_0}
(\boldsymbol{\beta}-\boldsymbol{\hat\beta})\,.\nonumber\end{align}
\end{proof}

\begin{remark}[A]\label{rem:bounds}
If, instead of $\mathcal N$, we take standard $\ell^2$ in $\mathbb R^{ML}$ (which correspond to the largest group in $\mathcal N$), the argument leading to \eqref{eq:18-10-sc3} yields \begin{align*}\sum_{\ell=L_0+1}^L|\boldsymbol{\hat\beta}^m_\ell|&=\sum_{\ell=L_0+1}^L|\boldsymbol{\beta}^m_\ell-\boldsymbol{\hat\beta}^m_\ell|\\ &\leq 
3\sum_{\ell=1}^{L_0}|\boldsymbol{\beta}^m_\ell-\boldsymbol{\hat\beta}^m_\ell|\,.\end{align*}
\end{remark}
\bigskip

\section{Proof of Theorem \ref{lem:main1}}\label{lem:main1:proof}
\begin{proof}
Since Lemma \ref{lem:2-9-sc2} implies 
\begin{eqnarray*}{\mathcal N}_\star
\left(\frac 2DX^\top\boldsymbol{U}\right) &
\leq & \frac 1{\sqrt L}\left\Vert \frac 2DX^\top\boldsymbol{U}\right\Vert_\infty\\ 
&= & \frac 1{\sqrt L}\max_{m\in[M]}\max_{l\in [L]}\left\vert\frac 2D
\sum_{t=1}^{T_m} U^m_tX^m_{t-l}\right\vert
\,,\end{eqnarray*} it suffices, by an 
union bound argument, to find high-probability upper bound of the inner-product 
(see \eqref{notn11} for notations) \begin{eqnarray}\label{eq:4}\tau_{m,l}&\deq2\langle \boldsymbol{U}^m,X^{m,l}
\rangle&=2\left(\sum_{t=1}^{T_m}
U^m_tX^m_{t-l}\right)\,.\end{eqnarray} We now present standard techniques, as in \cite{MR3357870}, for example. One has 
\begin{eqnarray*}\tau_{m,l}&=&\left\Vert\boldsymbol{U}^{(m)}+X^{(m,l)}
\right\Vert^2_2-\left\Vert\boldsymbol{U}^{(m)}\right\Vert_2^2-
\left\Vert X^{(m,l)}\right\Vert^2_2\,.\end{eqnarray*} Hence, for any 
$\eta>0$, an union bound argument, together with mutual independence of the 
Gaussian random variable $U^m_j$ 
and the variables $X^m_{<j}$ for each $j\in [T_m]$, implies \begin{eqnarray*}&& \mathbb P
\left[\left\vert\frac{\tau_{m,l}}D\right\vert>\frac 32\eta\right] \\ &\leq & 
\mathbb P\left[\left\vert(\boldsymbol{U}^{(m)})^\top
(\boldsymbol{U}^{(m)})-\vare(U^m)\right\vert>{\frac D2\eta}\right]+
\mathbb P\left[\left\vert(X^{(m,l)})^\top
(X^{(m,l)})-\vare(X^m)\right\vert>{\frac D2\eta}\right]\\ 
&& \hspace{2.85cm}+~\mathbb P
\left[\left\vert(\boldsymbol{U}^{(m)}+X^{(m,l)})^\top
(\boldsymbol{U}^{(m)}+X^{(m,l)})-\vare(X^m+
U^m)\right\vert>{\frac D2\eta}\right]\,.\end{eqnarray*} Here, 
$\vare(Y)$ denotes the trace of the covariance matrix of 
the random variable $Y$. We 
now estimate each summand separately, starting with $$p_{m,l,\eta}\deq 
\mathbb P\left[\left\vert(\boldsymbol{U}^{(m)})^\top
(\boldsymbol{U}^{(m)})-\vare(U^{(m)})\right\vert>\frac D2\eta\right]\,.$$ By the running assumptions, 
the vector $\boldsymbol{U}^{(m)}$ is $T_m$-dimensional 
mean zero Gaussian having covariance matrix 
$\sigma_m^2\mathbb I_{T_m}$; by the inequality in 
proposition \ref{cor:25-10-sc20}, one has \begin{align}\label{eq:4.6}
p_{m,l,\eta} &\leq 2e^{-\frac{c_0D\eta}{8\sigma_m^2}\min\{1, \frac{D\eta}{8\sigma_m^2T_m}\}}\,. 
\end{align} Next, we 
consider $$q_{m,l,\eta}\deq
\mathbb P\left[\left\vert(\boldsymbol{X}^{(m,l)})^\top
(\boldsymbol{X}^{(m,l)})-\vare(X^{(m,l)})\right\vert>\frac D2\eta\right]\,.$$ 
Recall that the random vector $X^{(m,l)}$ is mean-zero Gaussian with covariance 
matrix $\Gamma^{(m)}$ mentioned in equation \eqref{eq:06} below. One 
has $$\vare(X^{(m,l)})=\tr(\Gamma^{(m)})=T_m\mathbb E[(X_0^m)^2]\,.$$ 
From lemma \ref{lem:06-09-sc1}, one has $|\!|\Gamma^{(m)}|\!|_{\oop}
\leq \epsilon^{-2}\sigma_m^2$; thus, proposition \ref{cor:25-10-sc20} yields 
\begin{align}\label{eq:41.7}
q_{m,l,\eta} &\leq 2e^{-\frac{c_0D\eta}{8\sigma_m^2\epsilon^{-2}}\min\{1, \frac{D\eta}{8\sigma_m^2T_m\epsilon^{-2}}\}} 
\,.\end{align} Finally, we 
consider $$r_{m,l,\eta}\deq\mathbb P
\left[\left\vert(\boldsymbol{U}^{(m)}+X^{(m,l)})^\top
(\boldsymbol{U}^{(m)}+X^{(m,l)})-\vare(X^m+
U^m)\right\vert>\frac D2\eta\right]\,.$$ Note 
that $\boldsymbol{U}^{(m)}+X^{(m,l)}$ is a mean zero Gaussian with 
symmetric covariance matrix $\tilde{\Gamma}^{(m)}$, whose $(t,s)$-entry (for $t\geq s$) 
is given by \begin{align*}\tilde{\Gamma}^{(m)}_{t,s}
&\deq\mathbb E[(X^m_{t-l}+U^m_t)(X^m_{s-l}+U^m_s)]\\
&=\mathbb E[(X^m_{t-l}X^m_{s-l}]+\mathbb E[X^m_{s-l}U^m_t]
+\mathbb E[X^m_{t-l}U^m_s]+\mathbb E[U^m_tU^m_s]\\
&=\mathbb E[(X^m_{t-l}X^m_{s-l}]
+\mathbb E[X^m_{t-l}U^m_s]+\sigma_m^2\boldsymbol{1}_{t=s}\,.\end{align*} Together 
with lemma \ref{lem:06-09-sc2}, this 
implies \begin{align*}|\!|\tilde{\Gamma}^{(m)}|\!|_{\oop}&
\leq \epsilon^{-2}\sigma_m^2+\epsilon^{-4}\sigma_m^2+\sigma_m^2\\
&=\sigma_m^2(1+\epsilon^{-2}+\epsilon^{-4})\,,\end{align*} and by 
proposition \ref{cor:25-10-sc20}, we fetch \begin{align}\label{eq:41.117}
r_{m,l,\eta} &\leq 2e^{-\frac{c_0D\eta}{8\sigma_m^2(1+\epsilon^{-2}+\epsilon^{-4})}\min\{1, \frac{D\eta}{8\sigma_m^2T_m(1+\epsilon^{-2}+\epsilon^{-4})}\}} 
\,.\end{align} Now suppose that 
$$\eta\geq \frac{8C_\sharp \sigma^2_{\max}(1+\epsilon^{-2}+\epsilon^{-4})
}{M}\,;$$ 
then, the inequalities in \eqref{eq:4.6}, \eqref{eq:41.7}, and \eqref{eq:41.117} 
are simplified as follows: 
\begin{equation}\label{eq:4.66}
\begin{aligned}
    p_{m,l,\eta} &\leq 2e^{-\frac{c_0D\eta}{8\sigma_m^2}}\,;\\ 
 q_{m,l,\eta} &\leq 2e^{-\frac{c_0D\eta}{8\sigma_m^2\epsilon^{-2}}}\;,\\ 
 r_{m,l,\eta} &\leq 2e^{-\frac{c_0D\eta}{8\sigma_m^2(1+\epsilon^{-2}+\epsilon^{-4})}}\,. 
\end{aligned}
\end{equation}
Of these, the right hand side is the largest in the bottom-most inequality \eqref{eq:4.66}. 
Note that, the union bound implies \begin{eqnarray*}\mathbb P
\left[\sup_{m,l}\left\vert\frac{\tau_{m,l}}D\right\vert>\frac 32\eta\right] &
\leq & \displaystyle\sum_{m,l}\mathbb P
\left[\left\vert\frac{\tau_{m,l}}D\right\vert>\frac 32\eta\right]\\ 
&\leq & \displaystyle\sum_{l\in [L]}\displaystyle\sum_{m\in [M]}\mathbb P
\left[\left\vert\frac{\tau_{m,l}}D\right\vert>\frac 32\eta\right]\\
\eqref{eq:4.66}~
\Rightarrow\hspace{1cm}&\leq & 6\displaystyle\sum_{l\in [L]}
\displaystyle\sum_{m\in [M]} e^{-\frac{c_0D\eta}{8\sigma_m^2(1+\epsilon^{-2}+\epsilon^{-4})}}\,.\end{eqnarray*} Thus, for any $\delta>0$, in order to have the inequality 
\begin{eqnarray*}\mathbb P\left[{\mathcal N}_\star
\left(\frac 2DX^\top\boldsymbol{U}\right)> 
\frac{3\eta}{2\sqrt L}\right]&\leq \delta\,,\end{eqnarray*} it suffices 
to have \begin{equation}\exp\left(-\frac{c_0D\eta}
{8\sigma_m(1+\epsilon^{-2}+\epsilon^{-4})}\right)\leq \frac{\delta}{ML}\,\end{equation} 
for each $m\in [M]$. Equivalently, it suffices to have the following for each $m\in [M]$: 
\begin{align*}{\frac{c_0D\eta}{8\sigma_m^2
(1+\epsilon^{-2}+\epsilon^{-4})}}&\geq \log\left(\frac{ML}{\delta}\right)\\ 
\mbox{equivalently,}\hspace{1cm}D&\geq \frac{8\sigma_m^2
(1+\epsilon^{-2}+\epsilon^{-4})}{c_0\eta}\log\left(\frac{ML}{\delta}\right)\,.\end{align*} 
This is equivalent to \begin{align}\label{disco}D&\geq \frac{8
\sigma_{\max}^2(1+\epsilon^{-2}+\epsilon^{-4})}{c_0\eta}\log\left(\frac{ML}{\delta}\right)\,.\end{align} It then follows 
from the argument above that 
\begin{equation}\mathbb P
\left[\frac 2D{\mathcal N}_\star(X
^\top\boldsymbol{U})
\geq \frac{3\eta}{2\sqrt L}\right]~\leq~ \delta\end{equation} holds, provided 
\begin{align*}D&\geq \frac{8
\sigma_{\max}^2(1+\epsilon^{-2}+\epsilon^{-4})}{c_0\eta}\log\left(\frac{ML}{\delta}\right)\,.\end{align*}
\end{proof}

\section{Proof of Proposition \ref{sc-07-11-1}}\label{sc-07-11-1:proof}
We first observe what happens in the $M=1$ case: this amounts to 
restricting $X$ to the top left block $$X_1\deq
\begin{pmatrix}x_{1-1}^1&\dots&x_{1-L}^1\\
\cdot & \dots & \cdot\\ \cdot & \dots & \cdot\\ 
x_{T_1-1}^1&\dots&x_{T_1-L}^1\end{pmatrix}\,.$$ 

\begin{lemma}[Blockwise concentration inequality]\label{0somelemma0}
For any integer $s_1>0$, write $\mathbb K(s_1)\deq\mathbb B_0(s_1)\cap \mathbb B_1(1)$. 
Then \begin{align}\label{sc-05-10-1}&\mathbb P\left[
\sup_{\boldsymbol{v}_1\in \mathbb K(s_1) 
}\left\vert \boldsymbol{v}_1^\top 
(X_1^\top X_1-T_1\Gamma^{(1)})\boldsymbol{v}_1
\right\vert\geq \frac{T_1^2\zeta_1\Lambda_{\max}(\Sigma_{\boldsymbol{U}_1})}
{2\mathfrak m(f)}\right]\nonumber\\ &\leq
2\exp\left(-\frac{T_1\min\{\zeta_1,\zeta_1^2\}}{2}+s_1\min\{\log L,~\log(21eL/s_1)\}\right)\,.\end{align}
\end{lemma}

\begin{proof}[Proof of Lemma \ref{0somelemma0}]
Let $\boldsymbol{v}_1\in\mathbb R^L$ be a fixed 
unit-normed vector. 
Then, $\boldsymbol{u}_1\deq X_1\boldsymbol{v}_1$ is a 
random mean-zero Gaussian vector in $\mathbb R^{T_1}$, 
with covariance matrix $\Sigma_{\boldsymbol{v}_1}^{X_1}\deq \mathbb E\left[
X_1\boldsymbol{v}_1\boldsymbol{v}_1^\top X_1^\top\right]$. One has 
\begin{align*}(\Sigma_{\boldsymbol{v}_1}^{X_1})_{r,s}
&=\sum_{j,l=1}^L\mathbb E\left[X^1_{r-j}\boldsymbol{v}_{1,j}
\boldsymbol{v}_{1,l}X^1_{s-l}
\right]\\ &=\sum_{j,l=1}^L\boldsymbol{v}_{1,j}
\mathbb E\left[X^1_{r-j}X^1_{s-l}\right]\boldsymbol{v}_{1,l}
\\ &=\boldsymbol{v}_1^\top \Gamma_{r,s}\boldsymbol{v}_1\,,
\end{align*} 
where $\Gamma_{r,s}$ is the covariance of the vectors 
$X^{(1)}_r$ and $X^{(1)}_s$. 
Note that \begin{align*}\tr(\Sigma_{\boldsymbol{v}_1}^{X_1})
&=\mathbb E\left[\tr(X_1\boldsymbol{v}_1
\boldsymbol{v}_1^\top X_1^\top)\right]\\
&=T_1\mathbb E\left[(X_1\boldsymbol{v}_1)_{11}^2\right]\\
&=T_1\boldsymbol{v}_1^\top\Gamma^{(1)}
\boldsymbol{v}_1\,,\end{align*} where $\Gamma^{(1)}$ is 
the autocovariance matrix of $\{X^1_t\}$. 
Moreover, one has 
by Lemma \ref{lem:06-09-sc1}, the inequality 
\begin{align*}|\!|\Sigma_{\boldsymbol{v}_1}^{X_1}|\!|_{\oop}
\leq \frac{\Lambda_{\max}(\Sigma_{\boldsymbol{U}_1})}
{\mathfrak m(f)}\,.\end{align*} 
By the inequality in Proposition \ref{cor:25-10-sc20}, 
one has \begin{align*}
\mathbb P\left[\left\vert \boldsymbol{u}_1^T\boldsymbol{u}_1-\tr(\Sigma_{\boldsymbol{v}_1}^{X_1})
\right\vert\geq 4T_1\zeta_1|\!|\Sigma_{\boldsymbol{v}_1}^{X_1}|\!|_{\oop}\right]& \leq   
2e^{-\frac{T_1}2\min\{\zeta_1,\zeta_1^2\}}\end{align*} for any $\zeta_1>
0$. In particular, this implies that for any fixed $\boldsymbol{v}_1\in\mathbb R^L$, 
the inequality \begin{align}\label{sc-04-10-2}\mathbb P\left[
\left\vert \boldsymbol{v}_1^\top (X_1^\top X_1-T_1\Gamma^{(1)})\boldsymbol{v}_1
\right\vert\geq \frac{4T_1\zeta_1\Lambda_{\max}(\Sigma_{\boldsymbol{U}_1})}
{\mathfrak m(f)}\right]& \leq   
2e^{-\frac{T_1}{2}\min\{\zeta_1,\zeta_1^2\}}\end{align} holds for any $\zeta_1>0$. 
Applying Lemma \ref{lem:08-10-sc2} with $G\deq X_1^\top X_1-T_1\Gamma^{(1)}$, 
we conclude that \begin{align}\label{sc-04-10-3}&\mathbb P\left[
\sup_{\boldsymbol{v}_1\in\mathbb K(s_1)}\left\vert \boldsymbol{v}_1^\top 
(X_1^\top X_1-T_1\Gamma^{(1)})\boldsymbol{v}_1
\right\vert\geq \frac{4T_1\zeta_1\Lambda_{\max}(\Sigma_{\boldsymbol{U}_1})}
{\mathfrak m(f)}\right]\nonumber\\ \leq &   
2\exp\left(-\frac{T_1\min\{\zeta_1,\zeta_1^2\}}{2}+s_1\min\{\log L,~\log(21eL/s_1)\}\right)\end{align} holds for any 
integer $s_1\geq 1$, and any $\zeta_1>0$.
\end{proof}

\subsection*{Proof of Proposition \ref{sc-07-11-1}}
\begin{proof}
When $s_1>0$ is even, it follows from Lemma \ref{lem:08-10-sc3} that 
\begin{align*}&\mathbb P\left[
\sup_{\boldsymbol{v}_1\in\mathbb R^L}\left\vert\boldsymbol{v}_1(X^\top_1X_1-
T_1\Gamma^{(1)})\boldsymbol{v}_1\right\vert
\leq 
\frac{108T_1\zeta_1\Lambda_{\max}(\Sigma_{\boldsymbol{U}_1})}
{\mathfrak m(f)}\left(|\!|\boldsymbol{v}_1|\!|_2^2
+\frac 2{s_1}|\!|\boldsymbol{v}_1|\!|_1^2\right)\right]\\ \geq~ &   
1-2\exp\left(-\frac{T_1\min\{\zeta_1,\zeta_1^2\}}{2}+s_1\min\{\log L,~\log(21eL/s_0)\}\right)\,.\end{align*} An 
application of Lemma \ref{lem:06-09-sc1} implies that for 
$$\delta\deq 2e^{-\frac{T_1\min\{\zeta_1,\zeta_1^2\}}{2}+s_1\min\{\log L,~\log(21eL/s_1)\}}\,,$$ 
the conclusion of the following sequence of inequalities holds for all $\boldsymbol{v}_1\in
\mathbb R^L$ with probability at least $1-\delta$: 
\begin{align}\label{sc-11-10-1}
|\!|X_1\boldsymbol{v}_1|\!|_2^2
&\geq T_1\lambda_{\min}(\Gamma^{(1)})|\!|\boldsymbol{v}_1|\!|_2^2
-\frac{108T_1\zeta_1\Lambda_{\max}(\Sigma_{\boldsymbol{U}_1})}
{\mathfrak m(f)}\left(|\!|\boldsymbol{v}_1|\!|_2^2
+\frac 2{s_1}|\!|\boldsymbol{v}_1|\!|_1^2\right)\nonumber\\
&\geq \frac{T_1\Lambda_{\min}(\Sigma_{\boldsymbol{U}_1})}
{\mathfrak M(f)}|\!|\boldsymbol{v}_1|\!|_2^2-\frac{108T_1\zeta_1\Lambda_{\max}(\Sigma_{\boldsymbol{U}_1})}
{\mathfrak m(f)}\left(|\!|\boldsymbol{v}_1|\!|_2^2
+\frac 2{s_1}|\!|\boldsymbol{v}_1|\!|_1^2\right)\nonumber\\
&=T_1\left(\frac{\Lambda_{\min}(\Sigma_{\boldsymbol{U}_1})}
{\mathfrak M(f_1)}-\frac{108\zeta_1\Lambda_{\max}(\Sigma_{\boldsymbol{U}_1})}
{\mathfrak m(f_1)}\right)|\!|\boldsymbol{v}_1|\!|_2^2-
\frac{216T_1\zeta_1\Lambda_{\max}(\Sigma_{\boldsymbol{U}_1})}
{s_1\mathfrak m(f_1)}|\!|\boldsymbol{v}_1|\!|_1^2\,.\nonumber\end{align} 
We specialize to the case of univariate $\epsilon$-stable autoregressive 
time series, with $\Sigma_{\boldsymbol{U}_1}=\sigma_1^2$, 
and (by $\epsilon$-stability) $$\epsilon^2\leq \mathfrak m(f_1)\leq \mathfrak M(f_1)\leq \epsilon^{-2}.$$ 
Letting $\zeta_1\deq 6^{-3}\epsilon^4$, we obtain 
\begin{eqnarray}\label{eq:sc-11-10-011}&&\mathbb P\left[\inf_{\boldsymbol{v}_1
\in\mathbb R^L}\left(|\!|X_1\boldsymbol{v}_1|\!|_2^2
- \frac{T_1}2\sigma_1^2\epsilon^2\left(|\!|\boldsymbol{v}_1|\!|_2^2-
\frac2{s_1}|\!|\boldsymbol{v}_1|\!|_1^2\right)\right)\geq 0\right]\nonumber\\
&\geq& 1-2e^{-\frac{T_1}2 \zeta_1^2
+s_0\min\{\log L,~\log(21eL/s_0)\}}\,.\end{eqnarray} 
Finally, we consider all the blocks of $X^\top X$; 
the deviation inequality above yields
\begin{eqnarray}\label{eq:121}&&\mathbb P\left[\displaystyle
\forall_{m\in[M]}\inf_{\boldsymbol{v}_m\in\mathbb R^L}
\left(|\!|X_m\boldsymbol{v}_m|\!|_2^2- 
\frac{T_m\sigma^2\epsilon^2}
2\left(|\!|\boldsymbol{v}_m|\!|_2^2-\frac{2}
{s_m}|\!|\boldsymbol{v}_m|\!|_1^2\right)\right)\geq 0\right]\nonumber\\
&\geq& 1-2\sum_{m=1}^Me^{-\frac{T_m\zeta^2}{2}
+s_m\min\{\log L,~\log(21eL/s_m)\}}\,.\end{eqnarray}
This proves the proposition.
\end{proof}

\medskip





\section{More on Dual Norms}\label{Dua_norm}
Since ${\mathcal N}_\star(\boldsymbol{0})=0$, we 
will assume, in the following computation 
of the dual norm, that $\boldsymbol{\alpha}\neq \boldsymbol{0}$. This ensures ${\mathcal N}_\star(\boldsymbol{\alpha})> 0$ too. Note that $\mathcal N_\star(\boldsymbol{\alpha})
=\mathcal N_\star(\boldsymbol{|\alpha|})$, where $|\boldsymbol
{\alpha}|=(|\alpha_1|,\cdots,|\alpha_L|)$; this is because 
${\mathcal N}$ is invariant under arbitrary 
sign changes of the coordinates of its argument. Thus, 
for $\mathcal N_\star(\boldsymbol{\alpha})=
\boldsymbol{\alpha}\cdot \boldsymbol{\beta}$, we may assume 
(without loss of generality) that $|\boldsymbol{\alpha}|=\boldsymbol{\alpha}$ 
and $|\boldsymbol{\beta}|=\boldsymbol{\beta}$. 

The advantage of the hierarchical group norm ${\mathcal N}(\boldsymbol{\beta})$ --- in the setting of sparse recovery via regularized least square regression --- is that, while it sets any group of parameters to zero --- because of the hierarchy 
in the group structure --- all variables in all groups that appear further down the order of hierarchy are set to zero automatically. Thus, these norms are well-suited for determination of the true lag order in the context of autoregressive process.

In order to use this group norm in our analysis, we need to collect some 
basic facts on the norm. Note that the norm resembles $\ell^1$ at the group level, while within each group it is the $\ell^2$ norm; from this, one might expect that the dual of the norm should ``resemble" the $\ell^\infty$-norm at the group level, and within a group it should be $\ell^2$. We will prove that this intuition goes quite well, in the sense 
that the actual dual norm can be upper-bounded by this mixed 
$\ell^{\infty,2}$ norm.

\begin{lemma}[Dual norm is attained on the boundary]\label{claim:11}
\begin{eqnarray*}{\mathcal N}_\star(\boldsymbol{\alpha})&=&\max_{{\mathcal N}(\boldsymbol{\beta})=1}
\langle \boldsymbol{\alpha},\boldsymbol{\beta}\rangle\,.\end{eqnarray*}
\end{lemma}

\begin{proof}[Proof of Lemma \ref{claim:11}]
Note that --- by continuity of $\boldsymbol{\beta}\mapsto 
\langle \boldsymbol{\alpha},\boldsymbol{\beta}\rangle$, compactness 
of the subset $C_{\boldsymbol{\beta}}\deq\{\boldsymbol{\beta}
:{\mathcal N}(\boldsymbol{\beta})\leq 1\}$, and since $
{\mathcal N}_\star(\boldsymbol{\alpha})\neq 0$ --- there 
is nonzero $\boldsymbol{\beta}_0\in C_{\boldsymbol{\beta}}$ such that 
${\mathcal N}_\star(\boldsymbol{\alpha})=\langle 
\boldsymbol{\alpha},\boldsymbol{\beta}_0\rangle$. 
Suppose, if possible, that $\boldsymbol{\beta}_0$ satisfies ${\mathcal N}(\boldsymbol{\beta}_0)<1$. 
In $(0,+\infty)$, the function $c\mapsto {\mathcal N}(c\boldsymbol{\beta}_0)$ 
is strictly increasing (continuous) function; thus, there is 
$c>1$ such that $c{\mathcal N}(\boldsymbol{\beta}_0)=
{\mathcal N}(\boldsymbol{c\beta}_0)\leq 1$; 
one has \begin{align*}\max\{\langle \boldsymbol{\alpha},\boldsymbol{\beta}\rangle
:{{\mathcal N}(\boldsymbol{\beta})\leq 1}\}&\geq \langle\boldsymbol{\alpha},
c\boldsymbol{\beta_0}\rangle\\ &=c\cdot \langle\boldsymbol{\alpha},
\boldsymbol{\beta}_0\rangle \\ & > \langle\boldsymbol{\alpha},
\boldsymbol{\beta}_0\rangle\\ & = {\mathcal N}_\star(\boldsymbol{\alpha})\,,
\end{align*} which is a contradiction.
\end{proof}

In order to facilitate our computations of the dual norm in $\mathbb R^{ML}$, we 
start with the special case of $M=1$; after 
Proposition~\ref{claim:sc-02-09-1}, we will extend this to 
the general case. Now, we define $${\mathcal N}^1(
\boldsymbol{\alpha})\deq\sum_{l=1}^L\sqrt{(L-l+1)\sum_{j=l}^L\alpha_j^2}\,.$$ Let us write $\boldsymbol{\alpha}_{\geq j}=(\boldsymbol{0},\alpha_j,\alpha_{j+1},
\cdots,\alpha_L)$ for $j\in [L]$. From now on, we assume 
that $$\prod_{j\in[L]}\alpha_j\neq 0\,.$$

We have the following proposition. 

\begin{proposition}[$L^\infty$ norm bounds dual norm, $M=1$ case]\label{claim:sc-02-09-1}
The following inequality holds for $\boldsymbol{\alpha}\in\mathbb R^L$: 
\begin{align*}{\mathcal N}^1_\star(\boldsymbol{\alpha})
&\leq L^{-\frac 12}|\!|\boldsymbol{\alpha}|\!|_\infty\,.\end{align*}
\end{proposition}

\begin{proof}[Proof of Proposition \ref{claim:sc-02-09-1}]
We apply the elementary inequalities 
\begin{align*}(\sqrt{L-l+1}|\!|\beta_{\geq l}|\!|_2- \sqrt{L-l+1}
|\beta_l|)^2&\geq (L-l+1)|\!|\beta_{\geq l}|\!|_2^2- (L-l+1)
|\beta_l|^2\\ &=(L-l+1)\sum_{j>l}\beta_j^2
\\ &=\left(\sqrt{(L-l+1)\sum_{j>l}\beta_j^2}\right)^2\\
\mbox{Cauchy-Schwarz}~\Rightarrow\hspace{1cm}
&\geq \left(\sum_{j>l}\beta_j\right)^2\,,\end{align*} to derive 
\begin{align*}{\mathcal N}^1(\boldsymbol{\beta})&=
\sum_{l=1}^L\sqrt{L-l+1}|\!|\beta_{\geq l}|\!|_2\\
&\geq \sum_{l=1}^L\left(\sqrt{L-l+1}|\beta_l|+\sum_{j>l}|\beta_j|\right)\\
&= \sum_{l=1}^L\left(\sqrt{L-l+1}+l-1\right)|\beta_l|\\
&\geq \sqrt L|\beta_l|\,.\end{align*}
Since ${\mathcal N}^1_\star(\boldsymbol{\alpha})\leq |\!|
\boldsymbol{\alpha}|\!|_\infty {\mathcal N}^1(
\boldsymbol{1})$, where $\boldsymbol{1}\in\mathbb R^L$ 
is the \textit{all-one} vector, it suffices to show that 
${\mathcal N}^1_\star(\boldsymbol{1})\leq L^{-\frac 12}$. Since 
$${\mathcal N}^1_\star(\boldsymbol{1})=\max_{\mathcal N^1(\boldsymbol
{\beta})= 1}\sum_{j\in [L]}\beta_j\,,$$ and 
\begin{align*}1&=\mathcal N^1(\boldsymbol
{\beta})\\ &=\sum_{j\in [L]}\sqrt{L-l+1}|\!|\beta_{\geq j}|\!|_2\\ 
&\geq\sqrt L\sum_{j\in [L]}|\beta_j|\,,\end{align*} the claim follows.
\end{proof}

With the above proposition dealing with the case $M=1$, we can formulate in general the following proposition.

\begin{proposition}[$L^\infty$ norm bounds dual norm, general case]\label{lem:2-9-sc2}
Suppose $\boldsymbol{\alpha}\in \mathbb R^{ML}$, and 
\begin{align*}{\mathcal N}(\boldsymbol{\alpha})&=
\sum_{l=1}^L\sqrt{M(L-l+1)}|\!|\boldsymbol{\alpha}_{(\geq l)}|\!|_2\,,\\
\mbox{where}\hspace{1cm} |\!|\boldsymbol{\alpha}_{(\geq l)}|\!|_2:
&=\sqrt{\sum_{j=l}^L\sum_{m=1}^M\alpha^2_{(m-1)+j}}\,.\end{align*} 
Then ${\mathcal N}_\star(\boldsymbol{\alpha})\leq L^{-\frac 12}
|\!|\boldsymbol{\alpha}|\!|_\infty$.
\end{proposition}

\begin{proof}[Proof of Proposition \ref{lem:2-9-sc2}]
Again, it suffices to show that ${\mathcal N}(\boldsymbol{
\alpha})\leq L^{-\frac 12}$ when $\boldsymbol{
\alpha}$ is the all-one vector.\\

Now, ${\mathcal N}_\star(\boldsymbol{
\alpha})$ is the optimum value of the problem 
\begin{align*}\mbox{maximize}\hspace{0.5cm}&\sum_{j\in [ML]}\beta_j\\
\mbox{subject to}\hspace{0.5cm}&{\mathcal N}(\boldsymbol{
\beta})=1\,.\end{align*} Write $\boldsymbol{
\beta}_{(l)}\deq(\beta_l,\beta_{L+l},\cdots,
\beta_{(M-1)L+l})^\top$ and $x_l\deq\sqrt M|\!|\boldsymbol{
\beta}_{(l)}|\!|_2$ for $l\in [L]$; by Cauchy-Schwarz, 
we have \begin{align*}\sum_{j\in [ML]}\beta_j &
\leq \sum_{l=1}^Lx_l\,.\end{align*} Moreover, $
{\mathcal N}(\boldsymbol{\beta})={\mathcal N}^1(\boldsymbol{x})$, where 
$\boldsymbol{x}\deq(x_1,\cdots,x_L)$; thus, the lemma 
follows by \ref{claim:sc-02-09-1}.
\end{proof}

\bigskip

\section{Some known results used in the proofs}\label{blah-blah1} 
The following tail bound on the norm of Gaussian random vectors is well-known; see \citep[Theorem 6.2.1]{vershynin} for details. 

\begin{proposition}[Wright-Hansen inequality]\label{cor:25-10-sc20}
There is an absolute constant $c_0>0$ such that the following statement holds. If $\boldsymbol{q}\sim N(\boldsymbol{0},\Sigma)$ in $\mathbb R^d$, where $\Sigma$ is symmetric positive definite, then the following holds for any $\tau>0$: \begin{align*}\mathbb P\left[\frac 1d\left\vert \boldsymbol{q}^T\boldsymbol{q}-\tr(\Sigma)\right\vert \geq 4\tau|\!|\Sigma|\!|_{\oop}\right]&\leq 2e^{-c_0d\min\{\tau^2, \tau\}}\,.\end{align*}
\end{proposition}

We need to collect some information about the time series model in equation \eqref{eq:0}. 
Let $T\deq\min_{m\in [M]}T_m$, and recall that we have written $D=T_1+\cdots+T_M$. Let $\Gamma^{(m,l)}$ be the covariance matrix of the random vector $X^{(m,l)}$; this is a symmetric matrix, and for any $r,s\in \{1-l,\dots,T_m-l\}$ with $r\leq s$, the stationarity of the time series yields \begin{equation}\label{eq:06}\Gamma^{(m,l)}_{r,s}=\mathbb E[X^{(m,l)}_{r-l}X^{(m,l)}_{s-l}]=\mathbb E[X^m_0X^m_{s-r}]\,.\end{equation} Thus, we 
can write $\Gamma^{(m)}\deq\Gamma^{(m,l)}$ for the ``auto-covariance" matrix.


Now, let further, for any 
two integers $T>L>0$, $ _{T,L}\Gamma$ denote 
the $T_1L\times T_1L$ matrix, whose $(r,s)$-th block 
is the covariance of the vectors $(X_{r-1},X_{r-2},\dots,X_{r-L})$ 
and $(X_{s-1},X_{s-2},\dots,X_{s-L})$. 

The following lemma gives eigenvalue bounds of these block covariance matrices. 

\begin{lemma}[Eigenvalue bound for Toeplitz matrices]
\label{lem:06-09-sc1}
Let $f(z)$ denote the reverse characteristic polynomial of 
a stable univariate AR-process $\{X_t\}$. 
Let $\mathfrak m(f)=\inf_{|z|\leq 1}|f(z)|^2$ and 
$\mathfrak M(f)=\sup_{|z|\leq 1}|f(z)|^2$. The following inequality 
holds: $$\frac{\Lambda_{\min}(\Sigma_\epsilon)}
{\mathfrak M(f)}\leq \Lambda_{\min}( _{T,L}\Gamma)
\leq \Lambda_{\max}( _{T,L}\Gamma)\leq \frac{\Lambda_{\max}(\Sigma_\epsilon)}
{\mathfrak m(f)}\,.$$
\end{lemma}

\begin{proof}[Proof of Lemma \ref{lem:06-09-sc1}]
This is immediate from \citep[Proposition 2.3]{MR3357870} 
if we consider the stable $\dim$-L VAR process $Y_t\deq
(X_{t-1},\dots,X_{t-L})$.
\end{proof}

The following bound appeared as  
of \citep[Proposition 2.4b]{MR3357870}.

\begin{lemma}[Basu-Michailidis]\label{lem:06-09-sc2}
Let $\Delta^{(m)}$ denote the $T_m$-dimensional matrix 
with $\Delta_{t,s}^{(m)}\deq\mathbb E[X^m_{t-l}U^m_s]$. The following inequality 
holds: $$\Lambda_{\max}(\Delta^{(m)})\leq \frac{\Lambda_{\max}(\Sigma_{\epsilon_m})
\mathfrak M(f_m)}{\mathfrak m(f)}\,.$$
\end{lemma}

The following appeared as \citep[Lemma F.2]{MR3357870}. Recall that $\mathbb M_L(\mathbb R)$ denotes the algebra of all $L\times L$ real matrices, and $\mathbb S^{L-1}$ the unit sphere in $\mathbb R^L$.

\begin{lemma}[Basu-Michailidis]\label{lem:08-10-sc2}
Suppose that $G\in \mathbb M_L(\mathbb R)$ is a random symmetric matrix for which the following inequality holds for every $\boldsymbol{u}\in\mathbb S^{L-1}$, $T_1\in \mathbb N$, and $\eta>0$: \begin{equation*}
\mathbb P\left[\left\vert\boldsymbol{u}^TG\boldsymbol{u}\right\vert\geq C\eta\right]\leq
2e^{-cT_1\min\{\eta,\eta^2\}}\,,
\end{equation*} where $C,c>0$ are parameters independent of $\eta$ and $\boldsymbol{u}$. Then the following inequality holds for any integer $s_0>0$: \begin{align}\label{sc-08-10-11} \mathbb P\left[\sup_{|\!|\boldsymbol{u}|\!|_0\leq s_0,~ |\!|\boldsymbol{u}|\!|_2 \leq 1} \left\vert \boldsymbol{u}^TG\boldsymbol{u}\right\vert\geq C\eta\right]& \leq 2e^{-cT_1\min\{\eta,\eta^2\}+s_0\min\{\log L,\log(21eL/ s_0)\}}
\,.\end{align}
\end{lemma}

The following result was obtained as \citep[Lemma 12]{MR3015038}. Recall that, for any integer $s_0>0$, we denote $$\mathbb K(s_0)=
\{\boldsymbol{v}\in\mathbb R^L:|\!|\boldsymbol{v}|\!|_2\leq 1,
~|\!|\boldsymbol{v}|\!|_0\leq s_0\}\,.$$

\begin{lemma}[Loh-Wainwright]\label{lem:08-10-sc3}
Suppose that $G\in \mathbb M_L(\mathbb R)$ is a symmetric matrix for which the following holds: \begin{equation*}
\sup_{|\!|\boldsymbol{u}|\!|_0\leq s_0,~|\!|\boldsymbol{u}|\!|_2\leq 1} \left\vert \boldsymbol{u}^TG\boldsymbol{u} \right\vert \leq \delta
\,.\end{equation*} Then \begin{equation}\label{sc-08-10-12}
    \sup_{\boldsymbol{u}\in\mathbb R^L} \left\vert \boldsymbol{u}^TG\boldsymbol{u}\right\vert \leq 27\delta\left(|\!|\boldsymbol{u}|\!|_2^2+\frac 2{s_0}|\!|\boldsymbol{u}|\!|_1^2\right)
\,.\end{equation}
\end{lemma}






\end{appendix}

\bigskip

\end{document}